\documentclass[10pt,a4paper]{amsart}
\usepackage{hyperref}
\usepackage{mathrsfs}
\usepackage{amsmath}
\usepackage{amssymb}
\usepackage[all]{xy}
\usepackage{latexsym}
\usepackage[utf8]{inputenc}
\usepackage[T1,T5]{fontenc}
\title[Derived functors]{On the derived functors of destabilization and of 
iterated loop functors}
\author[G. Powell]{Geoffrey Powell}
\address{Laboratoire angevin de recherches en mathématiques (LAREMA),
CNRS, Université d’Angers, Université Bretagne Loire, 
2 Bd lavoisier 49045 Angers Cedex 01}
\email{Geoffrey.Powell@math.cnrs.fr}
\keywords{Steenrod algebra -- unstable module -- destabilization -- iterated 
loop functor -- derived functor -- total Steenrod power}
\subjclass[2000]{Primary 55S10; Secondary 18E10}
\date{}
\newtheorem{thm}{Theorem}[subsection]
\newtheorem{prop}[thm]{Proposition}
\newtheorem{cor}[thm]{Corollary}
\newtheorem{lem}[thm]{Lemma}
\theoremstyle{definition}
\newtheorem{defn}[thm]{Definition}
\newtheorem{exam}[thm]{Example}
\newtheorem{exer}[thm]{Exercise}
\newtheorem{hyp}[thm]{Hypothesis}

\theoremstyle{remark}
\newtheorem{rem}[thm]{Remark}
\newtheorem{nota}[thm]{Notation}
\newtheorem{conj}[thm]{Conjecture}
\newcommand{\lcx}{\mathfrak{C}}
\newcommand{\conn}{\mathsf{conn}}
\renewcommand{\hom}{\mathrm{Hom}}
\newcommand{\ext}{\mathrm{Ext}}
\newcommand{\unstalg}{\mathscr{K}}
\newcommand{\dash}{\mbox{-}}
\newcommand{\chcx}{\mathrm{Ch}}
\newcommand{\dhom}{\mathbb{D}}
\newcommand{\dcx}{\mathfrak{D}}
\newcommand{\largetensor}{\underline{\underline{\otimes}}}

\newcommand{\amod}{\mathscr{M}}
\newcommand{\unst}{\mathscr{U}}
\newcommand{\nil}{\mathscr{N}il}

\renewcommand{\phi}{\varphi}
\renewcommand{\epsilon}{\varepsilon}
\newcommand{\nat}{\mathbb{N}}
\newcommand{\zed}{\mathbb{Z}}
\newcommand{\field}{\mathbb{F}}

\newcommand{\op}{^\mathrm{op}}
\newcommand{\cala}{\mathscr{A}}

\newcommand{\colim}[1]{\underset{{\scriptstyle #1}}{\mathrm{colim}\ }}
\newcommand{\tor}{\mathrm{Tor}}
\begin{document}
\setcounter{tocdepth}{1}
\begin{abstract}
These notes explain how to construct  small functorial chain complexes which 
calculate the derived functors 
of destabilization (respectively iterated loop functors) in the theory of 
modules over the mod $2$ Steenrod algebra;  
this shows how to unify results of Singer and of Lannes and Zarati. 
\end{abstract}

\maketitle
\section{Introduction}

These  notes  consider the interface between unstable
modules $\unst$ over the
Steenrod algebra $\cala$ and $\amod$, the category 
of $\cala$-modules, in particular the structure of the left derived functors of
destabilization 
$D: \amod \rightarrow \unst$ (which is the left adjoint to the inclusion $\unst
\subset \amod$) and of the family of 
iterated loop functors, $\Omega^t : \unst \rightarrow \unst$ (the left adjoint
to the suspension functor 
 $\Sigma^t : \unst \rightarrow \unst$), for $t \in \nat$. For clarity of 
exposition, the prime 
$2$ is privileged,  and the underlying field $\field$ is usually the prime field 
$\field_2$; there are 
however analogous results for odd primes.

The motivation comes from the study of mod $2$ singular cohomology. The 
Steenrod 
algebra $\cala$ is the algebra of stable cohomology operations 
for mod $2$ cohomology, hence the cohomology of a spectrum (i.e. an object from 
{\em stable} homotopy theory) is naturally an $\cala$-module. The cohomology of 
a 
{\em space}
has more structure: it is an {\em unstable} $\cala$-module and is equipped with 
the cup product. The suspension on spectra (or pointed topological spaces) 
corresponds via reduced cohomology to  the algebraic suspension; 
on $\amod$ this is an equivalence of categories but not on unstable modules 
$\unst$. The algebraic loop functor $\Omega^t : \unst \rightarrow \unst$ is a 
first approximation to the behaviour in cohomology of the iterated loop functor 
on pointed topological spaces. Similarly, the destabilization functor  $D: 
\amod 
\rightarrow \unst$ gives a first algebraic approximation to the behaviour in 
cohomology of the infinite loop space functor $\Omega^\infty$ from spectra to 
pointed spaces. (This is  an over-simplification, since it takes no account of 
the cup product; see Section \ref{subsect:motiv} for slightly more precision.)

The two basic algebraic  ingredients which are used are the Singer functors 
$R_s$, $s \in 
\nat$,  which
are
defined for all $\cala$-modules, and the Singer residue map $\field [u^{\pm 1}]
\rightarrow \Sigma^{-1}\field$, which is
$\cala$-linear and induces differentials. Part of the interest of the current
approach is that it provides 
a clear explanation of the relationship between the methods of Lannes and Zarati
\cite{lannes_zarati_deriv_destab} and those of Singer 
(\cite{singer_new_chain_cx} etc.). 
The work of Lannes and Zarati makes no allusion to chain complexes; from the 
current viewpoint, they are considering the degenerate case where the 
differential is trivial. 
It is worth noting that their argument makes essential usage of connectivity, 
which is also a key point in the strategy used here.

This text explains how to construct a natural chain complex $\dcx M$, for $M$ an
$\cala$-module, with homology 
calculating the derived functors of destabilization, and, for $t \in \nat$ and
$N$ an unstable module, a chain complex
$\lcx^t N$, with homology calculating the derived functors of $\Omega^t$. The
existence of such a  
chain complex goes back to the work of Singer
\cite{singer_loops_I,singer_loops_II}, but the construction given 
here is new.

The complex $\lcx ^t N$ is given as a quotient of $\dcx  (\Sigma^{-t}N)$
and the 
projection $$\dcx (\Sigma^{-t}N) \twoheadrightarrow \lcx^tN$$ induces  in 
homology the
natural transformation 
$
D_s (\Sigma^{-t}N) \rightarrow \Omega^t_s N
$
between left derived functors. 

The chain complex $\dcx M$ is also related to the chain complex $\Gamma^+M$
introduced by Singer \cite{singer_invt_lambda} 
and Nguy{\~\ecircumflex}n H. V. H{\uhorn}ng  and  Nguy{\~\ecircumflex}n Sum 
\cite{hung_sum} (who work at odd primes), to calculate the
homology of $M$ over the Steenrod algebra. 
Namely, there is a natural inclusion 
\[
 \dcx M
\hookrightarrow 
\Gamma^+M
\]
which, in homology, induces the Lannes-Zarati homomorphism (up to dualizing)
\cite{lannes_zarati_deriv_destab},
the derived form of:
\[
 D M 
\rightarrow 
\field \otimes_\cala M, 
\]
thus giving rise to $D_s M \rightarrow \mathrm{Tor}^\cala _s (\field, M)$. (A
word of warning: $DM$ is an unstable module, in particular concentrated
in degrees $\geq 0$, hence the map to $\field \otimes _\cala M$ is not in
general surjective.) This morphism is of interest, since it is intimately 
related to the mod $2$ Hurewicz morphism.

The final section indicates some recent developments in the subject and some 
open problems. 
 In particular, the potential higher chromatic analogues of this theory are 
 of significant interest. This material is only outlined  here.

\bigskip
A number of exercises (of varying levels of difficulty) are included, 
reflecting the origin of this text as lecture notes; likewise, some proofs are 
left as exercises. The reader is encouraged 
to attempt them all, since they are essential to the understanding of the 
subject.

\bigskip
{\bf Acknowledgement:} The author is grateful to the anonymous referee for 
their careful reading of the manuscript and for their suggestions.

The author would like to thank the VIASM for the invitation to give this 
lecture course and for providing excellent working 
conditions during his visit in August 2013. He is also grateful to the 
participants of the special activity on {\em Algebraic Topology} for their 
interest. 

Finally, the  author also wishes to acknowledge the financial support provided 
by the 
\href{http://viasm.edu.vn/?lang=en}{VIASM} and 
\href{https://www.math.univ-paris13.fr/formath/index.php/fr/liafv}{LIA 
Formath}.

\tableofcontents

\section{Background}

Throughout, unless indicated otherwise, the underlying prime is taken to be $2$ 
and the ground field 
$\field$ is 
the field $\field_2$ with two elements. 
All the results introduced have analogues for odd primes, although the arguments
are slightly more complicated in the odd primary situation.

A general reference for the theory of (unstable) modules over the Steenrod
algebra is the book by Schwartz \cite{schwartz_book} and, for $\cala$-modules, 
that of Margolis \cite{margolis}. 
References for the results stated can be found  for example in the 
author's
papers \cite{powell_mod_loop,p_destab,p_singer_Rs}; many go 
back to Massey and Peterson and the work of Singer.

\subsection{The Steenrod algebra as a quadratic algebra}

The mod $2$ Steenrod algebra $\cala$ is, by definition, the algebra of stable
cohomology
operations for  mod $2$ singular 
cohomology. Hence the Steenrod algebra can be identified with the cohomology
$H^* (H\field_2)$ (here $H^*(-)$ always
 denote cohomology with mod $2$ coefficients) of the 
Eilenberg-MacLane spectrum, $H\field_2$, which represents mod $2$ cohomology.

The algebra $\cala$ is a non-homogeneous quadratic algebra, as explained
below. Let $\tilde{\cala}$ be the  algebra 
\[
 \tilde{\cala} := T (Sq^i | i \geq 0) / \sim 
\]
where $Sq^i$ has degree $i$ and $\sim$ corresponds to the Adem relations (for 
$a <2b$):
\[
 Sq^a Sq^b = \sum_j^{[a/2]} \binom{b-j-1}{a-2j} Sq^{a+b-j} Sq^j,
\]
where $Sq^0$ is considered as an independent generator.
Since the relations 
are homogeneous of length $2$, the algebra $\tilde{\cala}$ is a homogeneous
quadratic algebra and, in particular, has a {\em length grading} in addition to
the 
internal grading coming from the degrees of the generators. 

There is a surjection of algebras 
$
\tilde{\cala}
\twoheadrightarrow \cala
$
which corresponds to imposing the relation $Sq^0 =1$; the algebra $\cala$
inherits a {\em length filtration}
from $\tilde{\cala}$ (no longer a grading). The relations defining $\cala$ have 
length $\leq 2$, which
means  that 
$\cala$ is quadratic. 

The graded $\overline{\cala}$ associated to the length filtration can also be
described 
as a quotient of $\tilde{\cala}$, namely
\[
 \overline{\cala} = \tilde{\cala}/ \langle Sq^0 \rangle.
\]
 This is again a homogeneous quadratic algebra. Moreover, it has the important
property that it is 
{\em Koszul}. This notion, introduced by Priddy \cite{priddy}, is at the
origin of the existence
of small resolutions for calculating the homology of the Steenrod algebra; the
Koszul dual is the (big) Lambda algebra. 

The construction of the complexes introduced here is related to the quadratic 
Koszul nature of
$\cala$ and also to the relationship between 
the Steenrod algebra and invariant theory; many of the ideas go back
to the work of Singer \cite{singer_loops_I,
singer_loops_II,singer_invt_lambda} etc.  

\begin{rem}
 The odd primary analogues depend upon the work of M{\`u}i
\cite{mui_cohom_operations,mui_mod_invt_symm}, which describes 
the (more complicated) relationship between invariant theory and the Steenrod
algebra.

See for example the work of Nguy{\~\ecircumflex}n H. V. H{\uhorn}ng  and  
Nguy{\~\ecircumflex}n Sum \cite{hung_sum} generalizing Singer's 
invariant-theoretic description of the Lambda algebra 
to odd primes, Zarati's generalization \cite{zarati_these} of his work with 
Lannes \cite{lannes_zarati_deriv_destab} and the author's paper 
\cite{p_destab}.
\end{rem}

\subsection{The category of $\cala$-modules}

Let $\amod$ denote the category of (left) $\cala$-modules. This is an 
abelian category with additional 
structure; namely, the fact that $\cala$ is a Hopf algebra implies that the
tensor
product (as graded vector spaces) of two $\cala$-modules has a natural
$\cala$-module structure. 
Explicitly, the Steenrod squares act via: 
\[
 Sq^n (x \otimes y ) = \sum_{i+j = n} Sq^i (x) \otimes Sq^j (y);
\]
this corresponds to the fact that the diagonal $\Delta : \cala \rightarrow \cala
\otimes \cala$ 
is determined by $\Delta Sq^n = \sum_{i+j = n} Sq^i \otimes Sq^j $.

Since $\cala$ is a connected algebra (concentrated in non-negative degrees, with
$\cala ^0 = \field$) 
the Hopf algebra conjugation (or antipode) $\chi : \cala^\circ \rightarrow 
\cala$ is
determined by the diagonal \cite{milnor_moore}
and is an isomorphism of algebras, where $\cala ^\circ $ is $\cala$ equipped
with the {\em opposite}
algebra structure ($\chi$ is an anti-automorphism of $\cala$).

Via  $\chi$, the category of left $\cala$-modules 
is equivalent to the category of right $\cala$-modules:  a right
$\cala$-module $M$ can be considered as a left $\cala$-module by setting 
 $ 
 a m := m \chi (a)
$
for $a \in \cala$ and $m \in M$.

Hence the category  $\amod$ has a duality functor:
\begin{eqnarray*}
 (-)^\vee : \amod \op &\rightarrow& \amod 
\\
M & \mapsto & M^\vee := \hom_\field (M , \field), 
\end{eqnarray*}
where the usual {\em right} $\cala$-module structure on $M^\vee$ is regarded as
a left structure via $\chi$. 

\begin{nota}
 For $n \in \zed$, let $\Sigma ^n \field$ denote the $\cala$-module 
$\field$ in degree $n$. 
\end{nota}

\begin{rem}
 Since $\cala$ is connected, $\{ \Sigma^n \field | n \in \zed \}$ gives a set 
of 
representatives of
isomorphism classes of the simple objects of $\amod$. 
\end{rem}

\begin{exam}
 Duality gives $(\Sigma^n \field) ^\vee = \Sigma^{-n} \field$.
\end{exam}

\begin{defn}
 For $n \in \zed$, the $n$th suspension functor  $\Sigma^n : \amod \rightarrow
\amod$ is  $\Sigma^n \field \otimes - $. 
\end{defn}

\begin{prop}
\label{prop:amod_projectives}
\ 
\begin{enumerate}
 \item 
  The category $\amod$ has enough projectives, with set of projective generators
$\{\Sigma ^n \cala | n \in \zed \}$. 
\item 
For $n \in \zed$, $\Sigma^n : \amod \rightarrow \amod$ is an exact functor
which is an equivalence of
categories, with inverse $\Sigma^{-n} : \amod \rightarrow \amod$. In particular,
$\Sigma^n $ preserves 
projectives.
 \end{enumerate}
\end{prop}

\begin{proof}
For the first point, $\Sigma^n \cala$ is a free $\cala$-module, hence 
projective (in fact, $\Sigma^n \cala$ is the projective cover of $\Sigma^n 
\field$). Moreover, for $M$ an $\cala$-module, 
 $\hom_{\amod} (\Sigma^n \cala, M) \cong M^n$, whence it follows that $\{\Sigma 
^n \cala | n \in \zed \}$ is a set of projective generators. 
 
The second statement is clear.  
\end{proof}

\subsection{Unstable modules and destabilization}

Whereas the cohomology of a spectrum (object from stable homotopy theory which 
represents a cohomology theory) 
is simply an $\cala$-module, the cohomology of a space has further structure; 
it 
is an algebra (via the cup product) and the underlying
 $\cala$-module is {\em unstable}.

\begin{defn}
 An $\cala$-module $M$ is unstable if $Sq^i x = 0$, $\forall i  > |x|$.
The full subcategory of unstable modules is denoted  $\unst \subset \amod$. 
\end{defn}

\begin{prop}
\label{prop:unst_abelian_tensor}
 The category $\unst$ is an abelian subcategory of $\amod$ and is closed under
the tensor product  $\otimes$ of $\amod$. 
\end{prop}

\begin{proof}
From the definition of instability, it is clear that a submodule (respectively 
quotient) of an unstable module is unstable. 
This implies that $\unst$ is an abelian subcategory of $\amod$. 

Closure under $\otimes$  is seen as follows. By definition, $ Sq^n (x \otimes y 
) = \sum_{i+j = n} Sq^i (x) \otimes Sq^j (y)$; 
if $n > |x\otimes y |$ and $ i + j=n$, then either $i >|x|$ or $j > |y|$, so 
that the right hand expression is zero, as required.
\end{proof}

\begin{rem}
 The duality functor $(-)^\vee : \amod \op \rightarrow \amod$ does not
preserve 
$\unst$, since the relation $Sq^0 = 1$ implies that an unstable module is
concentrated in degrees $\geq 0$. 
The dual $M^\vee$ of a module $M$ concentrated in degrees $\geq 0$ is
concentrated in degrees $\geq 0$ 
if and only if $M = M^0$; for example, the dual of $\Sigma \field$ is not
unstable.
\end{rem}

\begin{exam}
For $n \in \nat$, the suspension functor $\Sigma^n : \amod \rightarrow \amod$
restricts to an exact functor 
$\Sigma^n : \unst \rightarrow \unst$ (given by $\Sigma^n \field \otimes -$).
This is {\em not} an equivalence of categories if $n >0$.
\end{exam}

For later use, the following definition is recalled, which uses the tensor
product of $\unst$.

\begin{defn}
\label{defn:unstalg}
 An algebra in $\amod$ is a graded algebra such that the structure morphisms are
$\cala$-linear.
An unstable algebra $K$ is an unstable module which is a  {\em commutative}
algebra in $\amod$ (and hence in $\unst$) such that the Cartan condition
holds: $Sq ^{|x|} (x) = x^2$, $\forall x \in K$.
Unstable algebras form a category $\unstalg$, with morphisms the algebra
morphisms which are $\cala$-linear.  Forgetting the algebra structure 
yields a functor $\unstalg \rightarrow \unst$. 
\end{defn}

In a few places the terminology {\em nilpotent}, {\em reduced}, {\em 
nil-closed} will be used; for the convenience of the reader, the definition is 
recalled (see 
\cite{schwartz_book} for further details).

\begin{defn}
\label{defn:nilclosed}
 An unstable module $N$ is nilpotent if the operation $Sq_0$ (where $Sq_0 
(x):=Sq^{|x|} (x)$) acts locally nilpotently.

An unstable module $M$ is 
{\em reduced} if $\hom_{\unst}(N,M)=0$ for any nilpotent module $N$ and {\em 
nil-closed} if, in addition, 
$\ext_\unst ^1 (N,M)=0$ for all nilpotents $N$.
\end{defn}

\begin{rem}
\label{rem:nilclosure}
 If $M$ is reduced, there is a nil-closed unstable module $\overline{M}$ and 
an inclusion 
$M\hookrightarrow \overline{M}$ with nilpotent cokernel. The module 
$\overline{M}$ is unique up to isomorphism
and is called the {\em nil-closure} of $M$.

Such considerations arise when considering the localization  of $\unst$ away 
from the nilpotent unstable modules, namely the study 
of the quotient category $\unst /\nil$ (see \cite[Chapter 5]{schwartz_book}).
\end{rem}

The following gives the archetypal examples of nilpotent unstable modules: 

\begin{exam}
 For $M$ an unstable module, $\Sigma M$ is nilpotent.
\end{exam}

The notion of {\em destabilization} arises naturally through topological
considerations, for example 
when passing from {\em stable} homotopy theory (spectra) to {\em unstable}
homotopy theory (spaces).

\begin{defn}
 Let $D : \amod \rightarrow \unst$ be the {\em left} adjoint to the (exact)
inclusion functor $\unst \hookrightarrow \amod$.
\end{defn}
 
\begin{exer}
\label{exer:destab_explicit}
 For $M$ an $\cala$-module, show that the linear subspace 
\[
 BM:= \langle Sq^i (x) | i > |x| \rangle,
\]
as $x$ ranges over elements of $M$, is a sub $\cala$-module. Deduce that $DM
\cong M / BM$ is an $\cala$-module (unstable, by construction).

Namely, from the explicit construction,  if $f : M \rightarrow N$
is a morphism of $\cala$-modules with $N$ unstable, there is a natural
factorization:
\[
 \xymatrix{
M 
\ar@{->>}[d] 
\ar[r]^f
&
N
\\
M/BM .
\ar@{.>}[ru]
}
\]
\end{exer}

\begin{nota}
For $n \in \zed$, let $F(n)$ denote $D (\Sigma^n \cala)$; this  
 is the {\em free unstable module} on a generator of
degree $n$.
\end{nota}

\begin{prop}
\label{prop:unst_projectives}
The category $\unst$ has enough projectives and 
 $\{ F (n) | n \in \nat\}$ forms a set of projective generators.
\end{prop}

\begin{proof}
 By adjunction, for $N$ an unstable module, there are natural isomorphisms
 \[
  \hom_{\unst}(D (\Sigma^n \cala) , N) \cong \hom_{\amod} (\Sigma^n \cala, N) 
\cong N^n,
 \]
which show both that $D (\Sigma^n \cala)$ is projective and that these form a 
set of projective generators (recalling that an unstable 
module is necessarily concentrated in non-negative degrees).
\end{proof}

\begin{prop}
\label{prop:destab}
 The functor $D: \amod \rightarrow \unst$ is right exact (but not exact) and
preserves projectives.
\end{prop}

\begin{proof}
The functor $D$ is left adjoint to the exact forgetful functor $\amod 
\rightarrow \unst$, hence preserves projectives and is right exact. 

To see that  $D$ is not exact, consider the $\cala$-module $\mathscr{E}:= 
\Sigma^{-1} \tilde{H}^* (\mathbb{R}P^2)$. The latter lies in the 
non-split short exact sequence in $\amod$:
\[
 0
 \rightarrow 
 \Sigma \field 
 \rightarrow 
 \mathscr{E}
 \rightarrow 
 \field
 \rightarrow
 0
\]
where the classes are linked by $Sq^1$. Clearly $D \mathscr{E} = \field$ and 
$D \Sigma \field = \Sigma \field$, so 
applying $D$ exhibits the non-exactitude. 
\end{proof}

The functor $D$ can be used to define {\em division functors}. The most
important examples considered
here are the (iterated) loop functors.

\begin{defn}
 For $n \in \nat$, let $\Omega^n : \unst \rightarrow \unst$ denote the composite
functor 
$D \Sigma^{-n}$, where  $\Sigma^{-n}: \amod \rightarrow \amod$ is 
restricted to a functor 
 $\unst \rightarrow \amod$. (For $n=1$, $\Omega^1$ is denoted simply $\Omega$.)
\end{defn}

\begin{prop}
\label{prop:Omega_projectives}
 For $n \in \nat$, the functor $\Omega^n : \unst \rightarrow \unst$ is left
adjoint to $\Sigma^n : \unst \rightarrow \unst$; it is  
  right exact (but not exact for $n >0$) and preserves projectives.
\end{prop}

\begin{proof}
It is straightforward to check that $D \Sigma^{-n}$ is left adjoint to 
$\Sigma^n : \unst \rightarrow \unst$. Since the latter is exact, it follows 
that 
$\Omega^n$ is right exact and preserves projectives. Non-exactitude can be seen 
as in Proposition \ref{prop:destab}; for example,  consider applying 
the functor $\Omega$ to $\Sigma \mathscr{E}$ for $n=1$ (noting that 
$\mathscr{E}$ is not unstable). 
\end{proof}

\begin{exer}
 Show that, for $0 < n \in \nat$, $\Omega F(n) \cong F(n-1)$. 
\end{exer}

\begin{prop}
\label{prop:D_iterated_loops}
 For $n \in \nat$ there is a natural equivalence of functors 
\[
 \Omega^n D \cong D \Sigma^{-n} : \amod \rightarrow \unst.
\]
\end{prop}

\begin{proof}
For $N$ an unstable module and $M$ an $\cala$-module, there is a chain of
natural adjunction isomorphisms
\begin{eqnarray*}
 \hom_\unst (\Omega^n D M, N) 
 &\cong& 
 \hom _\unst (D M, \Sigma^n N) 
 \cong 
 \hom_\amod (M, \Sigma^n N)
 \\
 &\cong& 
 \hom_\amod (\Sigma^{-n} M,  N)
 \cong 
 \hom_{\unst} (D \Sigma^{-n} M, N), 
\end{eqnarray*}
from which the result follows.
\end{proof}

\begin{exer}
Let $M$ be an unstable module which is of finite type (i.e. $\dim (M^n) $ is
finite $\forall n$). 
Show that the functor $D (- \otimes M^\vee) $ is left adjoint to $M \otimes -
$. 

This left adjoint is usually referred to as the {\em division functor} by $M$  
and written $(-:M)$; see \cite{la} for general considerations
on such functors.
\end{exer}

\begin{exam}
 An important division functor which can be constructed by
using destabilization is Lannes'
$T$-functor $T:= (-: H^* (B \zed/2))$, so that, for $M$ an unstable module,
\[
 T M  = D \Big(M \otimes H^* (B \zed/2)^\vee\Big), 
\]
where $B \zed/2$ is the classifying space of the group $\zed/2$, which has the 
homotopy type of $\mathbb{R}P^\infty$. (For the structure of $H^* (B \zed/2)$, 
see
Example \ref{exam:e1} below.) The dual $H^* (B \zed/2)^\vee$ can be identified 
as the homology $H_* (B\zed/2)$, considered
as a {\em left} $\cala$-module via the conjugation $\chi$.
\end{exam}

\begin{rem}
 Lannes' $T$-functor is an essential tool in modern homotopy theory and is 
defined
for any prime $p$. It has good properties which make it accessible to 
calculation; 
 for example it is exact and commutes with tensor products. Moreover, the 
$T$-functor 
restricts to a functor on the category $\unstalg$ of unstable algebras.
An underlying fundamental algebraic fact is  
 that the cohomology of an elementary abelian $p$-group, $H^* (BV; \field_p)$, 
is injective in $\unst$.

The homotopical importance of $T$ stems from the fact that $T H^* (X)$ is an 
approximation 
to the calculation of $H^* (\mathrm{Map} (B\zed/p , X) )$. Indeed, there is a 
canonical comparison 
map $TH^* (X) \rightarrow H^* (\mathrm{Map} (B\zed/p , X) )$ that is given by 
adjunction 
from the morphism induced in cohomology by the evaluation map
\[
 \mathrm{Map} (B\zed/p , X) \times B \zed/p \rightarrow X.
\]
Under adequate hypotheses upon $X$, the comparison map  is an isomorphism.

The reader should consult \cite{schwartz_book} for this and much more, in 
particular the 
application of $T$-functor technology to Sullivan's fixed point set conjecture 
(see also \cite{la}).
\end{rem}

\subsection{Derived functors}

The abelian categories $\amod$ and $\unst$ both have enough projectives (by 
Propositions \ref{prop:amod_projectives} and \ref{prop:unst_projectives}), 
hence 
one can
do homological algebra 
in them. Recall that a projective resolution $P_\bullet$ of an object $M$ of an
abelian category is a complex 
of projectives 
\[
 \ldots \rightarrow 
P_s \rightarrow P_{s-1} \rightarrow \ldots \rightarrow P_1 \rightarrow P_0,
\]
with $P_s$ in homological degree $s$, and which has homology concentrated in
degree zero with  $H_0(P_\bullet)\cong M$. 
This will frequently be denoted by $P_\bullet \rightarrow M$, where the arrow
corresponds to the 
surjection $P_0 \twoheadrightarrow M$. 

\begin{rem}
If $M$ is an unstable module, there are two possible notions of projective
resolution: 
a projective resolution in $\unst$, $P_\bullet \rightarrow M$ (that is, by 
projectives in
$\unst$),  or a  resolution 
in $\amod$,  $F_\bullet \rightarrow M$, by free $\cala$-modules. 
\end{rem}

\begin{defn}
 For $s,n\in \nat$, let 
\begin{enumerate}
 \item 
$D_s : \amod \rightarrow \unst$ denote the $s$th left derived functor of $D :
\amod \rightarrow \unst$;
\item
$\Omega^n _s : \unst \rightarrow \unst$ denote the $s$th left derived functor of
$\Omega^n : \unst \rightarrow \unst$. 
\end{enumerate}
\end{defn}

Explicitly, if $F_\bullet \rightarrow M$ is a free resolution in $\amod$
of an $\cala$-module 
$M$, then $D_s M $ is the $s$th homology of the complex $D F_\bullet$. Note that
$DF_\bullet$ is a complex 
with each object $DF_s$ projective in $\unst$, by Proposition \ref{prop:destab}; 
it is a projective resolution
of $DM$ if and only if all the higher
derived functors $D_s M$ vanish.  

\begin{exer}
 Let $M$ be an $\cala$-module and suppose that there exists $t\in \nat$ such
that $\Sigma^t M$ is unstable 
(such a $t$ does not exist in general - see Remark \ref{rem:never_unstable} 
below). Show that, for all $s \in
\nat$, there exists
a natural morphism 
 $
 D_s M \rightarrow \Omega^t _s \Sigma^t M.
$ 
This exhibits the close relationship between derived functors of destabilization
and of iterated loop functors.
\end{exer}

\begin{exam}
\label{exam:e1}
 Derived functors of destabilization are highly non-trivial. For example we
consider 
a lower bound for $D_1 (\Sigma^{-1} \field)$ as follows.

Recall that $H^* (B\zed/2) \cong \field [u]$, where $|u|=1$; this is an unstable
algebra, and this fact  determines
its structure as an $\cala$-module. (Explicitly, the total Steenrod power 
$Sq^{\mathrm{T}}=\sum_{i \in \nat}Sq^i$ 
on $u$
is 
$Sq^{\mathrm{T}} (u) = u (1+u)$, via the Cartan condition and instability and 
this
determines the structure 
via the Cartan formula for cup products, which implies that $Sq^{\mathrm{T}}$ 
is 
multiplicative.)

One can form the localization $\field[u^{\pm 1}]$, so that $u^{-1}$ is a class
of degree $-1$. 
This has an $\cala$-algebra structure (not unstable!), which is determined 
by
the total Steenrod power 
of $u^{-1}$. This can be calculated by using the multiplicativity of 
$Sq^{\mathrm{T}}$:
\[
 1  = Sq^{\mathrm{T}}(1) = Sq^{\mathrm{T}}(u^{-1}u ) = Sq^{\mathrm{T}} (u^{-1}) 
Sq^{\mathrm{T}} (u),
\]
 giving $Sq^{\mathrm{T}} (u^{-1}) = \frac{u^{-1}} {1+u}$, which translates as 
\[
 Sq^{n+1} (u^{-1}) = u^n 
\]
for all $n \in \nat$. 

Let $\hat{P}$ denote the sub $\field[u]$-module of $\field [u^{\pm 1}]$
generated by $u^{-1}$, so that 
there is a (non-split) short exact sequence of $\cala$-modules:
\[
 0 \rightarrow 
\field [u] \rightarrow \hat{P} \rightarrow \Sigma^{-1} \field \rightarrow 0.
\]
It is straightforward to see that $D \hat{P} = 0$ and $D \Sigma^{-1} \field=0$. 
Hence the long exact sequence for 
derived functors 
\[
 \ldots \rightarrow D_1 \hat{P} \rightarrow D_1 (\Sigma^{-1} \field) \rightarrow
D \field[u] 
\rightarrow D \hat {P} \rightarrow D \Sigma^{-1}\field\rightarrow 0
\]
shows that $D _1 (\Sigma^{-1} \field) \twoheadrightarrow  D\field [u] \cong
\field [u]$ is surjective. 
(It is in fact an isomorphism, by Corollary \ref{cor:deriv_destab_LZ}.) Thus 
$D_1 (\Sigma^{-1}
\field)$ is {\em infinite}, even 
though $\Sigma^{-1}\field$ has total dimension one.
\end{exam}

\begin{rem}
\label{rem:never_unstable}
The $\cala$-module $\hat{P}$ is bounded below; however $\Sigma^{t}
\hat{P}$ is {\em never} unstable, since all the $Sq^i $ act non-trivially upon 
the lowest dimensional class.
\end{rem}

\subsection{Motivation for studying derived functors of destabilization and of
iterated loop functors}
\label{subsect:motiv}

The functors $D_s : \amod \rightarrow \unst$ arise naturally in algebraic
topology. 

\begin{exam}
 There is a Grothendieck spectral sequence  calculating $\ext _\cala^* (M,
N)$ in terms of $\ext_\unst$ when
$N$ is an unstable module. This is the spectral sequence derived from 
considering $\hom_\cala (-,N)$ as the composite functor 
$\hom_\unst (D(-), N)$.

The spectral sequence has the form 
\[
 \ext ^p_\unst (D_q M, N) 
\Rightarrow 
\ext^{p+q} _\cala (M, N). 
\]
When $N$ is injective in $\unst$ (for example $N= \field$ or $N= H^* (BV)$, for 
$V$ an elementary abelian $2$-group), the spectral
sequence degenerates to the isomorphism
\[
\ext_\cala^q (M, N)
\cong 
 \hom _\unst (D_q M, N ) 
 .
\]
Such $\ext$ groups are important for calculating the $E^2$-term of the Adams
spectral sequence. This motivated Lannes and Zarati's work on the derived 
functors of destabilization
\cite{lannes_zarati_deriv_destab} and is intimately related to an approach to 
the Segal conjecture for elementary abelian $p$-groups.
\end{exam}

\begin{exam}
 Derived functors of destabilization occur in studying the relationship between
the cohomology of a spectrum $E$ and the cohomology
 of the infinite loop space $\Omega^\infty E$ associated to $E$. 

Recall that there is an adjunction counit $\Sigma^\infty \Omega^\infty E
\rightarrow E$, where $\Sigma^\infty$ is 
the suspension spectrum functor. This gives rise to a commutative diagram in
$\amod$:
\[
 \xymatrix{
H^* (E) 
\ar[rr]
\ar[rd]
&
&
H^* (\Sigma^\infty \Omega^\infty E) 
\cong 
H^* (\Omega^\infty E)
\\
&
D H^* (E)
,
\ar@{.>}[ur]
}
\]
where the factorization exists 
since $H^* (\Omega^\infty E)$ is an unstable algebra and, in particular, an
unstable module. 

Recall from Definition \ref{defn:unstalg} that $\unstalg$ denotes the category 
of unstable algebras and that the
Steenrod-Epstein enveloping 
algebra functor $U : \unst \rightarrow \unstalg$ is left adjoint to the
forgetful functor $\unstalg \rightarrow \unst$. This is given explicitly by 
\[
 UM := S^* (M)/ \langle Sq^{|m|} m = m^2 \rangle,
\]
the quotient of the free commutative algebra on $M$ given by imposing the 
Cartan condition.

Hence, the above induces a morphism of unstable algebras:
\[
 U (D H^*(E)) \rightarrow H^* (\Omega^\infty E). 
\]

When $E=\Sigma^n  H\field_2$ is a suspension of  the mod $2$ Eilenberg-MacLane
spectrum, this is an isomorphism. However, in general it is far from being an
isomorphism (examples can be given by considering suspension spectra $E= 
\Sigma^\infty X$). 
 Haugseng and Miller \cite{haugseng_miller} have constructed a spectral 
sequence which, 
 in favourable circumstances, calculates $H^* (\Omega^\infty E)$ from $H^* (E)$, 
 
based on 
 a cosimplicial Adams resolution of $E$ constructed from Eilenberg-MacLane 
spectra. The $E_2$-term is expressed
  in terms of non-abelian derived functors of $U D$. In particular, they show 
how the derived functors of destabilization 
  intervene.

Kuhn and McCarty \cite{kuhn_mccarty} construct a spectral sequence to calculate 
$H^* (\Omega^\infty E)$ using  the Goodwillie-Arone tower associated to the 
functor 
$E\mapsto \Sigma^\infty \Omega^\infty E$. In addition, they exhibited an {\em
algebraic approximation} to $H^* (\Omega^\infty E)$,  which is expressed in 
terms of the derived functors
of destabilization. This generalizes earlier work of Lannes and Zarati 
\cite{lannes_zarati_hopf} for
suspension spectra.  
\end{exam}

\begin{exam}
 Similar considerations arise in giving an algebraic approximation to  $H^*
(\Omega^n X)$
 in terms of  $H^* (X)$, when $X$ is a pointed space. As a first approximation,
one
shows that there 
is a natural morphism of unstable algebras 
\[
 U (\Omega^n Q H^* (X)) \rightarrow H^* (\Omega^n X),
\]
where $Q : \unstalg _a \rightarrow \unst$ is the `indecomposables' functor, 
defined
on the category $\unstalg _a$ 
of augmented unstable algebras by $QK:= \overline{K}/\overline{K}^2$, where
$\overline{K}$ is the augmentation ideal. (The base point of $X$ induces 
the augmentation of $H^*(X)$.)  This can be shown to be an isomorphism for
Eilenberg-MacLane spaces but, in general, is far from 
being an isomorphism. 

Under suitable  hypotheses on the space $X$, in particular supposing that
the cohomology of 
$X$ is of the form $UM$ for some unstable module $M$, Harper and Miller
\cite{harper_miller} gave an 
algebraic approximation to $H^* (\Omega^n X)$, which is expressed in terms of
the derived functors
of certain iterated loop functors. It is expected that their results can be
generalized.
\end{exam}

\begin{rem}
 Note that the topological based loop space functor $\Omega$ is {\em right}
adjoint to the reduced suspension functor 
$\Sigma$ and the infinite loop space functor $\Omega^\infty$ is {\em right}
adjoint to the suspension spectrum 
functor $\Sigma^\infty$ (at the level of homotopy categories). The suspension
functor $\Sigma$ commutes with cohomology; however,  since $H^*(-)$ is
contravariant, the algebraic 
approximations to these functors are {\em left} adjoints.  
\end{rem}

\section{First results on derived functors of destabilization and of iterated 
loops}

In this section, elementary results on the derived functors of $D$ and
$\Omega^n$ are considered, 
as a warm-up to constructing the chain complexes which compute the respective
derived functors.
\subsection{Derived functors of $\Omega$}

There is a simple chain complex which calculates the derived functors of
$\Omega$; to define it requires
the introduction of the doubling (or Frobenius) functor $\Phi$. 

\begin{defn}
\label{defn:Phi_lambda}
 Let $\Phi : \amod \rightarrow \amod$ be the functor which associates to $M$ the
module $\Phi M $ concentrated
in even degrees with $(\Phi M)^{2k} = M^k$ and action of the Steenrod
algebra determined by 
$Sq^{2i} (\Phi x) = \Phi (Sq^i x)$. 

Let $\lambda _M : \Phi M \rightarrow M $ be the natural morphism (of graded
vector spaces) 
defined by $\lambda_M (\Phi x) = Sq_0 (x) := Sq^{|x|} (x)$. 
\end{defn}

\begin{rem}
The functor $\Phi$ and the linear transformation $\lambda$ are defined for {\em 
all} $\cala$-modules.
\end{rem}

\begin{prop}
If $M$ is an unstable module, $\lambda_M  : \Phi M \rightarrow M$ is
$\cala$-linear; hence $\lambda$ induces 
a natural transformation 
$
 \lambda : \Phi |_\unst \rightarrow 1_\unst.
$
\end{prop}

\begin{proof}
This is an important exercise in using the Adem relation for $Sq^a Sq^{|x|} 
(x)$ and the instability condition. By the latter, one may reduce to the case 
 $a \leq 2 |x|$; it remains to show that the right hand side is zero for 
$a$ odd whereas, for $a = 2j \leq 2|x|$, 
\[
 Sq^{2j} Sq_0 (x) = Sq_0 Sq^j (x).
\]
The reader should provide the details for themselves.
\end{proof}

\begin{rem}
 For $M$ an unstable module, $\lambda_M$ is injective if and only if $M$ is a
{\em reduced} unstable module. 
\end{rem}

\begin{prop}
The functor $\Phi : \amod \rightarrow \amod$ satisfies the following
properties: 
\begin{enumerate}
 \item 
$\Phi$ is exact;
\item
$\Phi$ commutes with tensor products,  in particular $\Phi (\Sigma
M) \cong \Sigma^2\Phi M$.  
\end{enumerate}
\end{prop}

\begin{proof}
This follows directly from the definitions. 
\end{proof}

\begin{rem}
In odd characteristic $p$, the corresponding functor $\Phi$ does {\em not} 
commute with tensor
products; behaviour of $\Phi \Sigma$ is complicated, whereas $\Phi \Sigma^2 
\cong \Sigma^{2p} \Phi$.
\end{rem}

\begin{exer}
\label{exer:Phi_unstalg}
For $K$ an unstable algebra, show that  $\lambda_K : \Phi K
\rightarrow K$ is a morphism of unstable algebras. 
If $K$ is reduced (equivalently has no nilpotent elements as a commutative
algebra),  show that 
$\Phi K $ identifies with the subalgebra of $K$ generated by the squares of
elements of $K$. For example $\Phi \field[u] \cong \field[u^2] \subset
\field[u]$.
\end{exer}

\begin{prop}
\label{prop:ker_coker_lambda}
 For an unstable module $M$, the higher derived functors $\Omega_s$, $s >1$ of
$\Omega$ are
trivial, (i.e. $\Omega_s =0$ $ \forall s >1$) and there is a natural exact
sequence in
$\unst$
\[
 0 \rightarrow \Sigma \Omega_1 M \rightarrow \Phi M 
\stackrel{\lambda_M} {\rightarrow} M \rightarrow \Sigma \Omega M \rightarrow 0.
\]
In particular, the complex in $\amod$:
\[
 \Sigma^{-1}\Phi M 
\stackrel{\Sigma^{-1}\lambda_M} {\longrightarrow}
\Sigma^{-1} M 
\]
has homology $\Omega M$ and $\Omega_1 M$ in homological degrees $0, 1$
respectively.
\end{prop}

\begin{proof}
 By definition, $\Omega M$ is the destabilization of $\Sigma^{-1}M$. Hence
(using the fact that $M$ is 
unstable), 
\[
 \Sigma \Omega M \cong M / \langle Sq^{|x|} x \rangle,
\]
which is precisely the cokernel of $\lambda_M$.

It is a standard fact that the free unstable modules $F(n)$ are reduced (for 
example this can be seen 
since $F(n)$ is a submodule of $H^* (BV_n)$, where $V_n$ is a rank $n$ 
elementary abelian $2$-group), hence 
$\lambda_P$ is a monomorphism 
if $P$ is a projective unstable module.  

Consider a projective resolution $P_\bullet \rightarrow M$ in $\unst$. By the
above property, the natural
transformation $\lambda$ gives rise to a short exact sequence of complexes:
\[
 0 
\rightarrow 
\Phi P_\bullet
\rightarrow 
P_\bullet
\rightarrow 
\Sigma \Omega P_\bullet
\rightarrow 
0.
\]
The functors $\Phi$ and $\Sigma$ are exact, hence the homology of $P_\bullet$
and $\Phi P_\bullet$ is concentrated in homological degree zero,
where it is respectively $M$ and $\Phi M$, whereas the homology of $\Sigma
\Omega P_\bullet$ is isomorphic
to $\Sigma \Omega_s M$ in homological degree $s$, by construction of the derived
functors. 

The long exact sequence in homology in low degrees gives the exact sequence
\[
 0 = H_1 (\Phi P_\bullet) 
\rightarrow 
H_1 (\Sigma \Omega P_\bullet) 
\rightarrow 
H_0 (\Phi P_\bullet) = \Phi M 
\stackrel{\lambda_M}{\rightarrow}
M= H_0 (P_\bullet) 
\rightarrow 
H_0(\Sigma \Omega P_\bullet) , 
\]
which shows that the kernel of $\lambda_M$ is isomorphic to $\Sigma \Omega_1 M$,
as required. 

In higher homological degree, it follows immediately that $\Omega_ s M
$ = 0 for $s>1$. The final 
statement is clear.
\end{proof}

\begin{cor}
\label{cor:omega1_reduced}
For $p=2$, an unstable module $M$ is reduced if and only if $\Omega_1 M=0$.
\end{cor}

\begin{proof}
By definition, $M$ is reduced if and only if $\lambda_M$ is injective, hence if 
and only if $\Sigma \Omega_1 M =0$. 
The latter condition is equivalent to $\Omega_1 M=0$, as required.
\end{proof}

\begin{exer}
\ 
\begin{enumerate}
\item
Give an example of a nilpotent unstable module $N$ such that $\Omega N$ is
reduced.
\item 
Show that $\Omega_1 N$  is  nilpotent  if $N$ is 
a nilpotent unstable module. (Hint: reduce to the case that $N$ is a 
finitely-generated unstable module and 
hence to the case that $N$ has a finite filtration with quotients that are 
suspensions of unstable modules. 
 By induction on the length of the filtration, hence reduce to the case of a 
suspension.)
\item 
Zarati \cite{zarati} showed that an unstable module $M$ (over $\field_2$) is
nil-closed if and only if $M$ and $\Omega M$ are both reduced.  Prove this. 

Hints: 
\begin{enumerate}
\item 
To show that $\Omega M$ of a nil-closed unstable module is reduced, 
 show that it suffices to consider the nilclosed injective unstable modules and 
 
hence reduce to considering the case $M= H^* (BV)$, $V$ an elementary abelian 
$2$-group. 
The rank one case is straightforward; use this together with the behaviour of 
$\Omega$ on tensor products
 to treat the general case by induction.
\item 
If $M$ is reduced, there is a short exact sequence associated to nil-closure 
(see Remark \ref{rem:nilclosure}): 
\[
 0
 \rightarrow 
 M
 \rightarrow 
 \overline{M}
 \rightarrow 
 Q 
 \rightarrow 
 0,
\]
with $Q$ a nilpotent unstable module and $\overline{M}$  nil-closed. 
Now use the exact sequence of derived functors of $\Omega$ together with the 
fact that $\Omega_1 Q$ is nilpotent.
\end{enumerate}
\end{enumerate}
\end{exer}

\subsection{Applications of $\Omega$ and $\Omega_1$}

\begin{prop}
\label{prop:grothendieck_ses}
 For $C_\bullet$ a chain complex of reduced unstable modules,  $\Omega 
C_\bullet$ 
has homology which fits into a natural short exact sequence:
\[
 0 \rightarrow 
\Omega H_s (C_\bullet) 
\rightarrow 
H_s (\Omega C_\bullet) 
\rightarrow 
\Omega_1 H_{s-1}(C_\bullet) 
\rightarrow 
0.
\]
\end{prop}

\begin{proof}
 Since each $C_n$ is reduced, the natural transformation $\lambda$ induces a
short exact sequence of complexes
\[
 0
\rightarrow 
\Phi C_\bullet
\rightarrow 
C_\bullet
\rightarrow 
\Sigma \Omega C_\bullet 
\rightarrow
0.
\]
Using the exactness of $\Phi$ and $\Sigma$ together with the naturality of
$\lambda$, the associated long 
exact sequence in homology is 
\[
 \ldots 
\rightarrow 
\Phi H_s (C_\bullet) 
\stackrel{\lambda}{\rightarrow}
H_s (C_\bullet) 
\rightarrow 
\Sigma H_s (\Omega C_\bullet) 
\rightarrow 
\Phi H_{s-1} (C_\bullet) 
\rightarrow 
\ldots \ .
\]
By Proposition \ref{prop:ker_coker_lambda}, the cokernel of
$\lambda_{H_s(C_\bullet)}$ is $\Sigma \Omega H_s(C_\bullet)$ and
its kernel $\Sigma \Omega_1 H_s (C_\bullet)$. Applying these identifications
for $s$ and $s-1$ respectively
gives the stated short exact sequence. 
\end{proof}

\begin{cor}
\label{cor:loop_groth_ses}
For $s, t \in \nat$ and $M$ an unstable module, there is a natural short exact
sequence 
\[
 0
\rightarrow 
\Omega \Omega^t_s M 
\rightarrow 
\Omega^{t+1}_s M
\rightarrow
\Omega_1 \Omega ^t_{s-1} M
\rightarrow
0.
\]
In particular 
\[
\Omega^t _s M = \left\{ 
\begin{array}{ll}
 0 & s >t \\
(\Omega_1 ) ^t M & s =t.
\end{array}
\right.
\]
\end{cor}

\begin{proof}
 Let $P_\bullet \rightarrow M$ be a projective resolution of $M$ in $\unst$ and
take $C_\bullet := \Omega^t P_\bullet$, which is a complex of projective
unstable modules (by Proposition \ref{prop:Omega_projectives}) and these are 
reduced. 
The homology of $C_\bullet$ is, by definition, $H_s (C_\bullet) = \Omega^t_s M$,
whereas 
$H_s (\Omega C_\bullet)= \Omega^{t+1}_s M$.  The
required short exact sequence 
is furnished by Proposition \ref{prop:grothendieck_ses}. 

The final statement is proved by a straightforward induction  upon $t$.
\end{proof}

\begin{exer}
Assuming Zarati's theorem that an unstable module $M$ (over $\field_2$) is
nil-closed if and only if $M$ and $\Omega M$ are both reduced, show that $M$ is 
nilclosed if and only if $\Omega^t_s
M =0$ for $s>0$ and $t\leq 2$.
\end{exer}

\begin{rem}
For natural numbers $t_1, t_2$ and an unstable module $M$,
there is a Grothendieck 
spectral sequence 
\[ 
 \Omega^{t_1}_p \Omega^{t_2} _q  M \Rightarrow \Omega^{t_1+t_2}_{p+q}M. 
\]
The short exact sequence of Corollary \ref{cor:loop_groth_ses} corresponds to 
the case $t_1=1$.
\end{rem}

Corollary \ref{cor:loop_groth_ses} leads to an estimation of the connectivity of
the modules $\Omega^t_s M$. 

\begin{defn}
 For $M $ an $\cala$-module, the connectivity of $M$, $\conn (M) \in \zed \cup
\{-\infty,\infty\}$, is 
\[
 \conn (M) := \sup \{ i | M_j = 0, \  \forall j \leq i \}.
\]
\end{defn}

\begin{lem}
\label{lem:conn_Phi}
 For $M$ an $\cala$-module, $\conn (\Phi M) = 2 \conn (M ) +1$.
\end{lem}

\begin{proof}
An immediate consequence of the definitions of $\Phi$ and of $\conn$.
\end{proof}

\begin{prop}
 For $s, k \in \nat$ and $M$ an unstable module:
\[
 \conn (\Omega^{s+k}_s M) \geq 2^s (\conn (M) - k ).
\]
\end{prop}

\begin{proof}
It is clear that $\conn (\Omega M ) \geq \conn (M) -1$ and, 
 by Lemma \ref{lem:conn_Phi}, $\conn (\Omega_1 M) \geq 2 \conn (M)$. 
 The general result is proved by induction upon $s$, using the Grothendieck
short exact sequence of Corollary \ref{cor:loop_groth_ses} for the inductive
step. 
\end{proof}

\begin{rem}
 Since an unstable module $M$ is always concentrated in non-negative degrees,
$\conn (M) \geq -1$, hence it is clear that the previous statement is not 
optimal.
\end{rem}
\subsection{Interactions between loops and destabilization}

Recall from Proposition \ref{prop:D_iterated_loops} that, 
for $t\in \nat$, there is a natural isomorphism between $\Omega^t D , D
\Sigma^{-t}: \amod \rightrightarrows 
\unst$. The following result is another application of Proposition
\ref{prop:grothendieck_ses}:

\begin{cor}
\label{cor:destab_groth_ses}
 For $M$ an $\cala$-module, there is a natural short exact sequence:
\[
 0
\rightarrow 
\Omega (D_s M) 
\rightarrow 
D_s(\Sigma^{-1}M ) 
\rightarrow
\Omega_1 (D_{s-1}M)
\rightarrow 
0.
\]
\end{cor}

\begin{proof}
 Let $F_\bullet \rightarrow M$ be a free resolution of $M$ (in $\amod$)
and take $C_\bullet = 
D F_\bullet$, which is a complex of projective {\em unstable} modules by 
Proposition \ref{prop:destab}. 

Proposition \ref{prop:D_iterated_loops} implies that $\Omega C_\bullet$ is
naturally isomorphic to  $D (\Sigma^{-1} F_\bullet)$; $\Sigma^{-1}F_\bullet$ is 
a projective
resolution of $\Sigma^{-1}M$, hence
the homology of $\Omega C_\bullet$ calculates the derived functors $D_s
(\Sigma^{-1}M)$, whereas the homology 
of $C_\bullet$ calculates the derived functors $D_s M$. The result follows
immediately from Proposition \ref{prop:grothendieck_ses}.
 \end{proof}

\begin{rem}
The module $\Omega_1 (D_{s-1} M ) $ is the obstruction to $\Omega (D_s M )
\rightarrow D_s (\Sigma^{-1}M) $ being an isomorphism.
This is zero if and only if $D_{s-1}M$ is reduced, by Corollary 
\ref{cor:omega1_reduced}. 
\end{rem}

\begin{rem}
For $m \in \nat$ and an $\cala$-module $M$,  there is a Grothendieck spectral 
sequence
\[
 \Omega^m _p D_q M
\Rightarrow 
D_{p+q} \Sigma^{-m}M.
\]
The short exact sequence of Corollary \ref{cor:destab_groth_ses} corresponds to 
the case $m=1$.
\end{rem}

\subsection{Connectivity for $D_s$}

The explicit identification of the destabilization functor $DM = M/BM$ (see 
Exercise \ref{exer:destab_explicit}) 
leads to the following result:

\begin{lem}
\label{lem:destab_connectivity}
 For $M $ an $\cala$-module, the natural surjection 
$M \twoheadrightarrow DM$ is an isomorphism in degrees $\leq 2 (\conn M +1) $. 
\end{lem}

\begin{proof}
 The lowest degree element (if it exists - i.e. if $\conn (M)$ is finite) of $M$
has degree   $\conn (M) +1$, hence the lowest degree element of $BM$
has 
degree at least  $2 (\conn (M) +1 )+1 $. The result follows. 
\end{proof}

The following statement is a general result for connected algebras, stated here
for the Steenrod algebra.

\begin{lem}
\label{lem:proj_res_conn}
 An $\cala$-module $M$ has a free resolution $F_\bullet \rightarrow M$ in
$\amod$ with
 $\conn (F_s) \geq  \conn (M) +s $.
\end{lem}

\begin{proof}
 An exercise for the reader.
\end{proof}

The following weak result is sufficient for the initial applications; a much
stronger result holds (combine Lemma \ref{lem:conn_dhom} with Theorem 
\ref{thm:deriv_destab}). 

\begin{prop}
\label{prop:conn_Ds}
 For $0< s \in\nat$ and $M$ an $\cala$-module
\[
 \conn (D_s M) \geq 2 (\conn M + s).
\]
\end{prop}

\begin{proof}
It is sufficient to treat the case $\conn (M)$ finite (the other cases are
clear), hence we may take a free resolution
$F_\bullet \rightarrow M$ 
as in Lemma \ref{lem:proj_res_conn}. Consider the natural surjection of
complexes $F_\bullet \twoheadrightarrow 
D F_\bullet$. For $s>0$, the portion pertinent to $H_s$ is 
\[
 \xymatrix{
F_{s+1}
\ar[r]
\ar@{->>}[d]
&
F_s 
\ar@{->>}[d] 
\ar[r]
&
F_{s-1}
\ar@{->>}[d]
\\
DF_{s+1}
\ar[r]
&
DF_s
\ar[r]
&
D F_{s-1}
,
}
\]
where the top row is exact and the homology of the bottom row (in degree $s$) is
$D_s M$, by definition. 
The vertical morphisms are all isomorphisms in degrees $\leq 2 (\conn M + s)$,
by the hypothesis on $F_\bullet$ 
together with Lemma \ref{lem:destab_connectivity}; the result follows. 
\end{proof}

\begin{nota}
 For $M$ an $\cala$-module and $c \in \zed$, let $M^{>c}$ denote the sub
$\cala$-module of elements of degree $>c$, so that $\conn (M^{>c}) \geq c$.
\end{nota}

There is a natural short exact sequence of $\cala$-modules
\begin{eqnarray}
 \label{eqn:truncate_ses}
0
\rightarrow 
M^{>c}
\rightarrow 
M
\rightarrow
M/M^{>c}
\rightarrow 
0
\end{eqnarray}
and natural inclusions $M^{>c+1} \hookrightarrow M^{>c}$ and surjections 
$M/M^{>c+1} \twoheadrightarrow M/M^{>c}$ such that 
\begin{eqnarray*}
 M & \cong & \colim{c\rightarrow -\infty} M^{>c}\\
M &\cong & \lim_{c \rightarrow \infty} M/M^{>c}.
\end{eqnarray*}

\begin{prop}
\label{prop:destab_inv_limit}
 For $M$ an $\cala$-module, $s \in \nat$ and $c \in \zed$, the natural morphism
\[
 D_s M \rightarrow D_s (M/M^{>c}) 
\]
is an isomorphism in degrees $\leq 2 ( c+ s-1)$. 

Hence the natural morphism 
\[
 D_s M \rightarrow \lim _{c \rightarrow \infty} D_s (M/M^{>c})
\]
is an isomorphism. 
\end{prop}

\begin{proof}
 Consider the long exact sequence for the derived functors $D_s$ associated to
the short exact sequence (\ref{eqn:truncate_ses}):
\[
 \ldots \rightarrow
D_s (M^{>c}) 
\rightarrow 
D_s M 
\rightarrow 
D_s (M/M^{>c}) 
\rightarrow 
D_{s-1}(M^{>c})
\rightarrow 
\ldots \ .
\]
 Proposition \ref{prop:conn_Ds} implies that $\conn (D_s (M^{>c}) ) \geq 2 (c+s)
$ and $\conn (D_{s-1}(M^{>c})) \geq 2(c+s-1)$.
The first statement follows immediately, implying the  second. 
\end{proof}

\begin{rem}
 Proposition \ref{prop:destab_inv_limit} implies that, to study the derived 
functors $D_s$,  it is sufficient 
 to 
consider $\cala$-modules $M$  which are {\em bounded above} (i.e. such that 
$M^{>c}=0$ for $c
\gg 0$).
\end{rem}

\subsection{Comparing $D_s$ and $\Omega^t_s$}
This section establishes a precise relationship between the derived functors of 
destabilization and of iterated loop functors.
(This material is slightly technical and is not required for the subsequent 
results, hence can be skipped on first reading.)

Throughout the section, $M$ is taken to be an iterated desuspension of an 
unstable module,
so that there exists $T \in \nat$ such that $\Sigma^t M$ is unstable $\forall t
\geq T$. 
 If $M \neq 0$, $\conn (M)$ is finite; by Lemma \ref{lem:proj_res_conn}, there
exists a free resolution 
of $M$ in $\amod$, $F_\bullet \rightarrow M$, with $\conn (F_s) \geq \conn 
(M)+s$.
Consider the free 
resolution $\Sigma^t F_\bullet $ of $\Sigma^t M$, for $t\geq T$. Then, by
construction, $D (\Sigma^t F_\bullet)$ 
is a complex of projective unstable modules which has homology $H_s (D \Sigma^t
F_\bullet) \cong D_s (\Sigma^t M)$ and, in particular, $H_0 (D \Sigma^t
F_\bullet)= \Sigma^t M$; 
Proposition \ref{prop:conn_Ds} implies that, for $s >0$, $\conn (H_s (D \Sigma^t
F_\bullet)) = \conn (D_s (\Sigma^t M)) \geq 2 (\conn M +s +t) 
\geq 2 (\conn M + t +1)$. 

\begin{rem}
 The hypothesis upon $T$ implies that $\conn M +T +1 \geq 0$. 
\end{rem}

The complex $D \Sigma ^t F_\bullet$ can be seen as an approximation to a
projective resolution (in unstable modules) of the unstable module $\Sigma^t M$.
More precisely, one has 
the following:
 
\begin{lem}
\label{lem:approx_proj_res}
Let $M$ be an $\cala$-module and $t \in \nat$ such that $\Sigma^t M$ is 
unstable.
 There is a short exact sequence of complexes of projectives in $\unst$:
\begin{eqnarray}
\label{eqn:approx_proj_res}
 0 
\rightarrow 
D \Sigma^t F_\bullet
\rightarrow 
P_\bullet
\rightarrow 
Q_\bullet
\rightarrow 0
\end{eqnarray}
such that 
\begin{enumerate}
 \item 
$P_\bullet$ is a projective resolution of $\Sigma^t M$ in $\unst$;
\item
$D \Sigma^t F_\bullet
\rightarrow 
P_\bullet$ induces an isomorphism on $H_0$;
\item
$Q_0 =0$ and, $\forall s$, $\conn (Q_s) \geq 2 (\conn M + t +1)$. 
\end{enumerate}
\end{lem}

\begin{proof} (Indications.)
This is proved by the algebraic analogue of adding cells  in the process of CW 
approximation. 
Starting from the morphism $D \Sigma^t F_\bullet \rightarrow \Sigma^t M$ 
(considered as a morphism of chain complexes in $\unst$), 
 one adds `cells' (free unstable modules) to the complex $D \Sigma^t F_\bullet$ 
to obtain a factorization 
 \[
  \xymatrix{
  D \Sigma^t F_\bullet 
  \ar[r]
  \ar@{>->}[d]
  &
  \Sigma^t M
 \\
 P_\bullet
 \ar@{.>}[ur]_\simeq
  }
 \]
with the dotted map a weak equivalence. By construction, $P_\bullet$ is a 
projective resolution of $\Sigma^t M$, the vertical map induces an isomorphism 
in $H_0$, moreover it is an inclusion with cokernel $Q_\bullet$ a complex of 
projective unstable modules.

Finally, the connectivity estimate for the homology of $D \Sigma^t F_\bullet$ 
gives a lower bound on the connectivity of the cells which need to be added, 
hence upon $Q_\bullet$.
\end{proof}

\begin{prop}
Let $M$ be an $\cala$-module and $t\in \nat$ such that $\Sigma^t M$ is unstable.
Then, for all $s \in \nat$, 
the natural 
morphism 
\[
 D_s M \rightarrow 
\Omega^t _s \Sigma^t M
\]
is an isomorphism in degrees $\leq 2 (\conn M +1) +t $.  
\end{prop}

\begin{proof}
 Consider $F_\bullet \rightarrow M$ as above and the short exact sequence
(\ref{eqn:approx_proj_res})  of Lemma \ref{lem:approx_proj_res}. 
 Applying the functor $\Omega^t$ and using the natural isomorphism $\Omega^t D
(\Sigma^t F_\bullet) 
\cong D F_\bullet$ given by Proposition \ref{prop:D_iterated_loops}, this yields
a short exact sequence
of complexes 
\[
 0 
\rightarrow 
D F_\bullet
\rightarrow
\Omega ^t P_\bullet
\rightarrow 
\Omega ^t Q_\bullet 
\rightarrow 
0,
\]
where the first morphism induces $D_s M \rightarrow \Omega^t _s \Sigma ^t M $ 
in 
homology. 

The connectivity condition on $Q_\bullet$ implies that $\conn (\Omega^t Q_s)
\geq 2 (\conn M + 1) +t$. The result follows from the long exact sequence in
homology.
\end{proof}

\begin{cor}
 For $M$ an $\cala$-module and $T\in \nat$ such that $\Sigma^T M$ is unstable,
there is a natural isomorphism
\[
 D_s M \cong 
\lim _{T \leq t\rightarrow \infty } \Omega^t _s \Sigma^t M
\]
and the inverse system stabilizes locally for $t \gg 0$ (i.e. in any given
degree).
\end{cor}

\begin{proof}
An exercise for the reader. 
\end{proof}

\begin{exer}
Let $M \in \amod$ be a module which is bounded below ($M^n=0$ for $n \ll 0$).
Show that, for fixed $s,d \in \nat$, 
there exist $c, T\in \nat $ such that $\Sigma^T (M/ M^{>c})$ is unstable and,
for all $t \geq T$,
\[
 (D_s M)^d \cong 
\big(
\Omega^t _s \Sigma^t (M/M^{>c})
\big)^d.
\]
\end{exer}

\section{Singer functors}

Singer introduced a series of functors which are indispensable for understanding
the derived functors of iterated loops and destabilization.
 This section  recalls the definition of (variants of) these.

\subsection{The unstable Singer functors $R_s$}

Following Lannes and Zarati \cite{lannes_zarati_deriv_destab}, we recall the
construction of the {\em unstable} Singer functors 
$R_s$, for $s \in \nat$; by convention $R_0$ is the identity functor $R_0 :
\unst \rightarrow \unst$. 

\begin{nota}
 For $s\in \nat$, let $D(s)$ denote the $s$th Dickson algebra,
which is defined as the algebra of invariants
\[
D(s):=  H^* (BV_s) ^{GL_s}, 
\]
where the action of the general linear group on the cohomology of the
classifying space $BV_s$  is induced by the
natural action on $V_s = \field^{\oplus s}$.
\end{nota}

The Dickson algebra $D(s)$ has underlying algebra the polynomial algebra $\field
[\omega_{s, i } | 0 \leq i \leq s-1]$, 
 where $\omega_{s,i}$ is the Dickson invariant of degree $2^{s}- 2^i$ (for
example, the top Dickson invariant, $\omega_{s,0}$, 
is the product of all non-zero classes in $H^1 (BV_s)$). (See \cite{wilk} for 
further details on the Dickson algebras.)

\begin{nota}
 For $K$ an unstable algebra, let $K\dash\unst$ denote the category of
$K$-modules in $\unst$; forgetting the 
module structure defines a functor $K\dash\unst\rightarrow \unst$. 
\end{nota}

An object of $K\dash\unst$ is an unstable module $M$ equipped with a $K$-module
structure such that the structure 
map $K \otimes M \rightarrow M $ is $\cala$-linear.

\begin{prop}
For $K$ an unstable algebra, the category $K\dash\unst$ has a unique abelian
structure such that $K\dash\unst \rightarrow \unst$ 
is exact. Moreover the tensor product of $K$-modules $\otimes _K$ defines a
tensor structure on $K\dash\unst$, with unit $K$ (i.e. 
 $K\dash\unst$ is a symmetric monoidal category $(K\dash\unst, \otimes_K, K)$).
\end{prop}

\begin{proof}
An unstable algebra $K$ is, in particular, a unital commutative monoid in 
$\unst$
 and the category $K\dash\unst$ is its category of modules. The result is 
standard and 
is left as an exercise for the reader.
\end{proof}

Recall that $H^* (BV_1) \cong \field [u]$, with $|u|=1$, has a canonical
unstable algebra structure.

\begin{defn}
\label{def:St_1}
 For $M$ an unstable module, let 
\begin{enumerate}
 \item 
$St_1 : \Phi M \rightarrow \field[u] \otimes M$ denote the linear map
(not $\cala$-linear)
 defined by $St_1 (x):= \sum u^{|x|-i}\otimes Sq^i (x)$;
\item
$R_1 M$ denote the sub $\field[u]$-module of $\field [u] \otimes M$ generated by
$\{ St_1 (x) | x \in M\}$. 
\end{enumerate}
\end{defn}

\begin{rem}
The above notation is adopted for typographical simplicity; strictly speaking, 
$St_1 (x)$ should be 
denoted $St_1 (\Phi x)$.
\end{rem}

\begin{prop} 
 \cite{lannes_zarati_deriv_destab}
 For $M$ an unstable module, the sub $\field[u]$-module $R_1M \subset \field 
[u]\otimes M$ is stable under
the $\cala$-action, hence $R_1$ defines a 
functor $R_1 : \unst \rightarrow \field[u]\dash \unst$. 
\end{prop}

\begin{proof}
The proof is left as an essential exercise for the reader. (Remark 
\ref{rem:relate_R1_D1} and Exercise \ref{exer:LZ}  below provide the ingredients.
Namely, consider $\Sigma \hat{P} \otimes M$ (using the notation of Remark \ref{rem:relate_R1_D1}); 
 Exercise \ref{exer:destab_explicit} shows that $B (\Sigma \hat{P} \otimes M)$ is a sub $\cala$-module; it suffices
 to identify this.)
\end{proof}

\begin{rem}
 Forgetting the $\field[u]$-module structure, $R_1$ is frequently considered as
a functor $R_1 : \unst \rightarrow \unst$. However, the $\field[u]
$-module structure is  important when considering iterated loop
functors in Section \ref{subsect:iterated_loops}.
\end{rem}

The functor $R_1$ has many remarkable properties, such as indicated in 
Proposition \ref{prop:Singer_1_properties} below. (The 
richness of the behaviour of these functors is further exhibited by localizing 
away from nilpotents; for this, 
see 
\cite{p_singer_Rs}.) 

\begin{prop}
\label{prop:Singer_1_properties}
(Cf. \cite{lannes_zarati_deriv_destab}.) Let $M, N$ be unstable modules.
\begin{enumerate}
 \item 
 The functor $R_1 : \unst \rightarrow \field [u] \dash\unst $ is exact; more 
precisely, the underlying $\field[u]$-module 
of $R_1 M$ is isomorphic to $\field[u] \otimes \Phi M$.
\item
There is a natural isomorphism of unstable modules
\[
 \field \otimes_{\field[u]} R_1 M \cong \Phi M;
\]
the canonical surjection is written $\rho_1 : R_1M \twoheadrightarrow \Phi M$
and there is a natural short exact sequence in 
$\field[u] \dash\unst$:
\[
 0
\rightarrow 
u R_1 M
\rightarrow 
R_1 M 
\rightarrow
\Phi M 
\rightarrow
0\]
\item
The functor $R_1$ preserves tensor products: there is a natural
isomorphism $R_1 (M \otimes N) \cong R_1 M \otimes _{\field[u]} R_1 N$.
\end{enumerate}
\end{prop}

\begin{proof}
 See \cite{lannes_zarati_deriv_destab} (or prove this as a non-trivial 
exercise).
\end{proof}

\begin{exer}
\label{exer:R1}
\ 
\begin{enumerate}
\item 
Show that the map $St_1$ is injective and hence deduce part (1) of Proposition 
\ref{prop:Singer_1_properties}.
 \item 
Show that the projection $\rho _1 : R_1 M \rightarrow \Phi M$ is compatible with
$\lambda _M$, namely the following diagram commutes:
\[
 \xymatrix{
R_1M 
\ar[d]_{\rho_1}
\ar@{^(->}[r]
&
\field[u] \otimes M
\ar[d]^{\epsilon}
\\
\Phi M
\ar[r]_{\lambda_M}
&
M,
}
\]
where $\epsilon$ is induced by the augmentation of $\field[u]$. 

(Hint: consider the composite around the top of the diagram, which is a 
morphism of $\field [u]$-modules.) 
\item
Show that the total Steenrod power $St_1$ is multiplicative, when $K$ is an
unstable algebra. Namely, for $x, y \in K$, 
\[
 St_1 (x y) = St_1 (x) St_1 (y)
\]
where the product on the right hand side is formed in the unstable algebra
$\field[u] \otimes K$.
\item
For $K$ an unstable algebra, show that $R_1 K$ is naturally an unstable algebra,
equipped with a natural inclusion 
$\field[u] \hookrightarrow R_1 K$, so that $R_1$ defines a functor $R_1 :
\unstalg \rightarrow \field[u]\downarrow \unstalg$ to the category 
$\field[u]\downarrow \unstalg$ of $\field [u]$-algebras in $\unstalg$.

 (Hint: show that $R_1 K$ is a sub unstable algebra of $\field[u] \otimes K$. 
 For the morphism $ \field[u] \hookrightarrow R_1 K$, apply $R_1$ to the unit 
$\field \rightarrow K$.)
\item
For $K$ an unstable algebra, show that $R_1$ induces a functor $R_1 : K
\dash\unst \rightarrow R_1 K \dash \unst$. 

(Hint: use the functoriality of $R_1$.)
\item
Determine the structure of $R_1 \field[u_2] \subset \field[u_1, u_2]$ and
identify it as the ring of invariants for the action of the upper triangular
subgroup $B_2 \subset GL_2$. (Here $B_2$ is isomorphic to the group $\zed/2$ 
and acts 
 by $u_1 \mapsto u_1$ and $u_2 \mapsto u_2 + u_1$. The ring of invariants can 
be  calculated directly.)
\end{enumerate}
\end{exer}

\begin{rem}
\label{rem:relate_R1_D1}
 Lannes and Zarati \cite{lannes_zarati_deriv_destab} showed that $R_1$ is
intimately related to destabilization.
Namely, the short exact sequence (see Example \ref{exam:e1})
$$0 \rightarrow \field[u] \rightarrow \hat{P}
\rightarrow \Sigma^{-1}\field \rightarrow 0$$
defines a non-trivial class $e_1 \in \ext_\cala ^1 (\Sigma^{-1}\field,
\field[u])$. For an unstable module 
$M$, tensoring gives  the short exact sequence 
\[
 0 \rightarrow \field[u]\otimes M \rightarrow \hat{P}\otimes M \rightarrow
\Sigma^{-1}M \rightarrow 0
\]
and the long exact sequence for derived functors of destabilization induces a
morphism 
\[
 \alpha_M : D_1 (\Sigma^{-1}M) \rightarrow D (\field [u] \otimes M)= \field[u]
\otimes M.
\]
Considering the case $M = \Sigma N$, for an unstable module $N$,  Lannes and 
Zarati observed that
$\alpha_{\Sigma N}$ induces a surjection $$D_1 N \twoheadrightarrow \Sigma R_1 N
\subset \Sigma \field[u] \otimes N.$$

In the case $N=\field$, Lannes and Zarati proved moreover that $D_1 \field \cong
\Sigma R_1 \field \cong \Sigma \field[u]$ (this follows directly from
the chain complex constructed in Section \ref{subsect:destab_chcx} below). 
Proposition \ref{prop:grothendieck_ses} shows that $D_1 (\Sigma^{-1} \field)
\cong \Omega D_1 \field$, which is therefore 
isomorphic to $\field[u]$.
\end{rem}

\begin{exer} \label{exer:LZ} (Cf. \cite{lannes_zarati_deriv_destab}.)
Prove the  result of Lannes and Zarati stated above that, for $N$ an unstable 
module, 
\[
 \alpha_{\Sigma N} : D_1 N \twoheadrightarrow \Sigma R_1 N
\subset \Sigma \field[u] \otimes N
\]
is surjective.

(Hint: using the fact that $N$ is unstable, show that $B (\Sigma \hat{P} 
\otimes N) \subset \Sigma P \otimes N$ 
and identifies with $\Sigma R_1 N \subset \field [u ] \otimes \Sigma N$. Here 
it suffices to consider 
$Sq^i (\Sigma u^{-1} \otimes x)$ for $i >|x|$.)
\end{exer}

\begin{rem}
 The functor $R_1$ has topological significance: let $X$ be a pointed
topological space and write $E\zed/2$ for the universal
cover of $B\zed/2$, which is an acyclic space equipped with a free
$\zed/2$-action. (An explicit model is given by $S^\infty = \colim{n 
\rightarrow 
\infty} S^n$,
with projection $S^\infty 
\rightarrow \mathbb{R}P^\infty$ induced by the $\zed/2$-Galois
coverings $S^n \rightarrow \mathbb{R}P^n$.)

 The diagonal of $X$ induces a $\zed/2$-equivariant
map $E\zed/2_+ \wedge X \rightarrow E\zed/2_+ \wedge X \wedge X$ (here $(-)_+$
denotes the addition of a disjoint basepoint) and passage to the
quotient by the $\zed/2$-action gives:
\[
 \Delta_2 : B\zed/2_+ \wedge X 
\rightarrow 
\mathfrak{S}_2 X := E\zed/2_+ \wedge_{\zed/2} (X \wedge X).
\]
$\mathfrak{S}_2 X $ is the quadratic construction on the pointed space $X$.

In  mod $2$ cohomology, this induces 
\[
 \Delta^*_2 : \tilde{H}^* (\mathfrak{S}_2 X) 
\rightarrow 
H^* (B\zed/2) \otimes \tilde{H}^* (X)
\]
and the image of $\Delta^*_2$ is $R_1 \tilde{H}^* (X)$. This is related to
the {\em construction} of the Steenrod  operations.
\end{rem}

The Singer functors can be iterated. For example, $R_1 R_1 : \unst \rightarrow
R_1 \field[u] \dash \unst$ (see Exercise \ref{exer:R1}) and $R_1 R_1 M$ is the
sub $R_1 \field[u]$-module of 
$\field [u_1, u_2] \otimes M$ which is generated by $St_2 (x):= St_1 (St_1
(x))$. 

\begin{nota}
For $M$ an unstable module, $s\in \nat$ and a fixed basis of $V_s \cong 
\field^{\oplus s}$, define
$St_s$  as a linear map
\[
 St_s : \Phi^s M \rightarrow H^*(BV_s) \otimes M
\]
inductively by $St_s = St_1 \circ St_{s-1}$. 
\end{nota}

\begin{rem}
Here, for precision, one should indicate the basis element used for each $St_1$ 
(cf. \cite{lannes_zarati_deriv_destab}). However, this
issue is resolved by the following result.
\end{rem}

\begin{lem}
 \cite{lannes_zarati_deriv_destab}
 For $M$ an unstable module, the linear map $St_s$ takes values in $D(s) 
\otimes 
M \subset H^* (BV_s) \otimes
M$, hence is independent of the choice of basis of $V_s$ 
used in the definition.
\end{lem}

By construction, the iterated Singer functor $R_1 ^{\circ s}$ comes equipped
with a natural inclusion 
$R_1 ^{\circ s} M \hookrightarrow H^* (BV_s) \otimes M$, which depends upon the
basis used in the construction. This dependency
is removed by the following definition:

\begin{defn}
 For $s \in \nat$, let $R_s : \unst \rightarrow D (s) \dash \unst$ be the
functor defined on an unstable module $M$ by
\[
 R_s (M):= \big ( D(s) \otimes M \big ) \cap R_1 ^{\circ s} M .
\]
\end{defn}

\begin{rem}
 The functor can be defined explicitly by taking $R_s M$ to be the sub
$D(s)$-module of $D(s) \otimes M$ generated by $St_s (x)$, $\forall x \in M$.
The advantage of the previous construction is that it implies immediately that
this
submodule is stable under   $\cala$. 
\end{rem}

\begin{rem}
 The {\em quadratic} nature of the construction is exhibited by the identity for 
$s\geq 2$:
\[
 R_s = \bigcap _{a+b+2=s} R_1^{\circ a} R_2 R_1 ^{\circ b}.
\]
This shows that the functors $R_s$ are determined by the {\em generating}
functor $R_1$ and the {\em relation} 
$R_2 \hookrightarrow R_1 R_1$. 
\end{rem}

\begin{exer}
Make the previous statement precise and prove it (hint:
consider generators for $GL_s$).
\end{exer}

\begin{exer}
For $0< s \in \nat$ and any inclusion $i_s: V_{s-1} \subset V_s$, show that the
canonical inclusions of the Dickson invariants 
fit into  a commutative diagram in $\unstalg$:
\[
 \xymatrix{
D(s) \ar@{^(->}[rr]
\ar@{->>}[d]
&&
H^* (BV_s) 
\ar[d]^{i_s^*}
\\
\Phi D(s-1) 
\ar@{^(->}[r]
&
D(s-1) 
\ar@{^(->}[r]
&
H^*(BV_{s-1}) .
}
\]
(Use Exercise \ref{exer:Phi_unstalg} for the first inclusion of the bottom row.)
In particular, there is a canonical surjection
of unstable algebras $D(s) \twoheadrightarrow \Phi D(s-1)$. 

Explicitly, show that $i_s^*$ maps $\omega_{s,0}$ to zero and $\omega_{s,i}
\mapsto \omega_{s-1, i-1}^2$ for $i >0$.
\end{exer}

\begin{exer}
For $M$ in $D(s-1)\dash\unst$, show that $\Phi M$ is  naturally an object of 
$\Phi
D(s-1)\dash\unst$ and hence, via the surjection
$D(s)\twoheadrightarrow \Phi D(s-1)$, in $D(s)\dash\unst$. (An analogous result
holds replacing the module categories such as $D(s)\dash\unst$ by the category
$D(s)\dash\amod$ of $D(s)$-modules in $\amod$.)
\end{exer}

\begin{prop}
 For $s\in \nat$ and unstable modules $M, N$, 
\begin{enumerate}
 \item  
$R_s: \unst \rightarrow D(s)\dash\unst $ is exact and commutes with tensor
products: 
$R_s (M \otimes N) \cong R_s M \otimes_{D(s)} R_sN$.
\item
The natural transformation $\rho_1$ induces a natural surjection 
$\rho_s : R_s \twoheadrightarrow \Phi R_{s-1}$ via the inclusion $R_s
\hookrightarrow R_1 R_{s-1}$ composed with $(\rho_1)_{R_{s-1}}$, which 
fits into a short exact sequence in $D(s)\dash\unst$:
\[
 0
\rightarrow
\omega_{s,0} R_s M 
\rightarrow
R_s M 
\rightarrow 
\Phi R_{s-1}M
\rightarrow 
0.
\]
\end{enumerate}
\end{prop}

\begin{proof}
See \cite{lannes_zarati_deriv_destab} or prove this as an exercise.  (Hint: for  
$M=\field$, the short exact sequence
corresponds to the natural projection $D(s)\twoheadrightarrow \Phi D(s-1)$.)
\end{proof}

\begin{rem}
\label{rem:alpha_M_s}
 The class $e_1 \in \ext^1_\cala (\Sigma^{-1}\field, \field[u])$ of Remark 
\ref{rem:relate_R1_D1} gives rise, via
Yoneda product, to the class 
$e_s \in \ext^s_\cala (\Sigma^{-s}\field, H^*(BV_s))$ and it is a fundamental
result of Singer's that this class is invariant 
under the action of $GL_s$ (see \cite{lannes_zarati_deriv_destab}, for example).

Standard methods of homological algebra (it is easier to think in terms
of derived categories) show that the functor
$D$ induces a natural morphism (for  $s \leq t \in \nat$)
\[
 \ext^s _\cala (M,N) \rightarrow \hom_\unst (D_t M, D_{t-s}N).
\]
Thus, the class $e_s$ induces a linear morphism (natural in the $\cala$-module 
$M$)
\[
\alpha_s^M : D_s (\Sigma^{-s}M ) \rightarrow D (H^* (BV_s) \otimes M).
\]
If $M$ is unstable, the right hand side is $H^* (BV_s) \otimes M$ and Lannes and
Zarati show that 
$\alpha_s^M$ induces a map 
\[
\alpha_s^M : D_s (\Sigma^{-s}M ) \rightarrow R_sM
\]
(this may also  be seen using the results of the next section). 
This exhibits the relationship between the Singer functor $R_s$ and the derived
functor of destabilization $D_s$.
\end{rem}

\subsection{Singer functors for $\amod$}
\label{subsect:Rs_amod}

The unstable Singer functors $R_s : \unst \rightarrow D(s)\dash\unst
\stackrel{\mathrm{forget}}{\longrightarrow}
\unst$ generalize to 
\[
 \xymatrix{
\amod 
\ar[r] ^(.4){R_s} 
&
D(s)\dash\amod 
\ar[r]^(.6){\mathrm{forget}}
&
\amod,
}
\]
where $D(s)\dash\amod $ is the category of $D(s)$-modules in $\amod$.

Recall that localization gives an inclusion $\field[u]\hookrightarrow
\field[u^{\pm 1}] $ of $\cala$-algebras. If the $\cala$-module $M$ is not 
unstable, 
then the Steenrod total power $St_1$ (see Definition \ref{def:St_1}) on $M$ does 
not take 
values in $\field[u] \otimes
M$; if $M$ is bounded above it takes values 
in $\field[u^{\pm 1}] \otimes M$ but, in the general case, it is necessary to
use a {\em large} tensor product $\largetensor$ (half-completed tensor product 
- 
see
\cite{p_destab}, for example)
so that $St_1$ is a linear map
\[
St_1 : \Phi M 
\rightarrow 
\field [u^{\pm 1}] \largetensor M.                         
                        \]
With this modification, $R_1$ is defined as in the unstable case, so that $R_1
M$ comes equipped with a canonical 
inclusion $R_1 M \hookrightarrow \field [u^{\pm 1}] \largetensor M$. Many of the
good properties of $R_1$ pass to this setting, 
in particular:

\begin{prop}
 The functor $R_1 : \amod \rightarrow \field[u]\dash\amod$ is exact.
\end{prop}

\begin{proof}
A generalization of Proposition \ref{prop:Singer_1_properties}. 
\end{proof}

The higher functors $R_s$ are constructed as before;  the large tensor
product leads to some technical issues. 

Localization inverting the top Dickson invariant gives a commutative diagram
of $\cala$-algebras:
\[
 \xymatrix{
D(s)
\ar[r]
\ar[d]
&
H^* (BV_s) 
\ar[d]
\\
D(s)[\omega_{s,0}^{-1}] 
\ar[r]
&
H^* (BV_s)[\omega_{s,0}^{-1}] 
}
\] 
which, in the case $s=1$, corresponds to $\field[u]\hookrightarrow \field[u^{\pm
1}]$. The localized Dickson algebra 
$D(s)[\omega_{s,0}^{-1}]$ is the appropriate generalization of $\field[u^{\pm
1}]$.

The general Singer functors 
$
 R_s : \amod \rightarrow D(s)\dash\amod
$, come equipped with 
a natural embedding for an $\cala$-module $M$
\[
 R_s M \hookrightarrow D(s)[\omega_{s,0}^{-1}] \largetensor M, 
\]
and are exact. Moreover, they can be constructed from iterates of $R_1$ by
imposing the quadratic relation $R_2$.

\begin{rem}
Care must be taken in considering the composition because of the large tensor
product; see \cite{p_destab} (which is written 
for the odd characteristic case, but the methods also apply over $\field_2$).
\end{rem}

As in the unstable case, one has the following fundamental short exact sequence:

\begin{prop}
\label{prop:ses_Rs_amod} \cite{p_destab}
 For $0<s \in \nat$ and $M$ an $\cala$-module, there is a natural short exact
sequence in $D(s)\dash\amod$:
\[
 0
\rightarrow
\Sigma^{-1} R_s \Sigma M
\rightarrow
R_s M 
\rightarrow 
\Phi R_{s-1} M
\rightarrow 
0.
\]
\end{prop}

\subsection{The Singer differential}

There is a new phenomenon when considering the Singer functors defined on
$\amod$, corresponding to the 
Singer differential.

\begin{prop}
\label{prop:residue}
 The residue, namely the unique non-trivial map of graded vector spaces:
\[
 \partial : \field[u^{\pm 1}] 
\rightarrow 
\Sigma^{-1}
\field ,
\]
 is $\cala$-linear. (Equivalently, $u^{-1}$ is not in the image of $Sq^i$,
$\forall i >0$).
\end{prop}

\begin{proof}
Since the Steenrod algebra $\cala$ is generated by $\{ Sq^{2^n} |0<  n \in 
\nat\}$, it suffices 
to show that $Sq^{2^n} (u^{-(2^n +1)}) =0$ for all $0< n\in \nat$. As in Example 
\ref{exam:e1}, the Steenrod total square acts via 
$Sq^{\mathrm{T}} (u^{-1}) = \frac{u^{-1}} {1+u}$, hence, by multiplicativity of 
$Sq^\mathrm{T}$ and using the hypothesis $n>0$,
\[
 Sq^{\mathrm{T}} (u^{-(2^n+1)})= \Big(\frac{u^{-1}} {1+u}\Big) 
\Big(\frac{u^{-2^n}}{1+ u^{2^n}}\Big) = \Big(\frac{u^{-1}} {1+u}\Big)  \big 
(u^{-2^n} + 1 + \ldots \big ).
\]
It follows  that the term in degree $-1$ is zero, as required.
\end{proof}

\begin{defn}
\label{def:dM}
For $M$ an $\cala$-module, let $d_M: R_1 M \rightarrow \Sigma^{-1} M $ denote
the composite natural transformation:
\[
 R_1 M 
\hookrightarrow
\field[u^{\pm 1}] 
\largetensor M
\stackrel{\partial \largetensor M}{\longrightarrow}
\Sigma^{-1}M.
\]
\end{defn}

\begin{exer}
Show that, if $M$ is unstable, then $d_M : R_1 M \rightarrow \Sigma^{-1}M$ is 
trivial.
\end{exer}

The following result is the basis for building the chain complex calculating the
derived functors of destabilization.

\begin{prop}
\label{prop:coker_SigmaD}
For $M$ an $\cala$-module, the cokernel of $\Sigma d_M : \Sigma R_1 M
\rightarrow M$ is $DM$. 
\end{prop}

\begin{proof}
This is left as a fundamental exercise for the reader.
\end{proof}

\section{Constructing chain complexes}

Recall from Section \ref{subsect:Rs_amod} that $R_s : \amod \rightarrow
D(s)\dash\amod$ is an exact functor and that there is a natural differential
$d_M: R_1 M \rightarrow \Sigma^{-1}M$ for $M$ an $\cala$-module (see Definition 
\ref{def:dM}). Moreover,
there is a natural inclusion 
$R_s \hookrightarrow R_{s-1}R_1$. These are the key ingredients to constructing
the chain complexes which calculate the derived 
functors of destabilization and of iterated loop functors.

\subsection{Destabilization}
\label{subsect:destab_chcx}

\begin{defn}
For $M$ an $\cala$-module and $1\leq s \in \zed$, let
 $d_{s,M} : R_s M \rightarrow R_{s-1}(\Sigma^{-1}M) $ denote the natural
morphism given as the composite:
\[
\xymatrix{
 R_s M \ar@{^(->}[r]& 
R_{s-1}R_1 M 
\ar[rr]^{R_{s-1}d_M}&&
R_{s-1} (\Sigma^{-1}M),
}
\]
so that $d_{1,M}$ identifies with $d_M$. 
\end{defn}

\begin{prop}
 \label{prop:d_squared}
For $M$ an $\cala$-module and $s \in \nat$, the composite 
\[
\xymatrix@C=5pc{
 R_{s+2} (M) 
\ar[r]^{d_{s+2,M}}
&
R_{s+1}(\Sigma^{-1}M) 
\ar[r]^{d_{s+1,\Sigma^{-1}M}}
&
R_s (\Sigma^{-2}M)
}
\]
is trivial. 
\end{prop}

\begin{proof}
(Indications. See \cite{p_destab} for a proof in odd characteristic; the 
method adapts to $\field_2$.) Using the
quadratic nature of the functors $R_s$, it is straightforward to reduced to the
case $s=0$.

This case can be proved using the relationship between the Steenrod algebra and
invariant theory, as in the work of Singer \cite{singer_invt_lambda}; 
one method is to embed the diagram in the $\field_2$-analogue of the chain 
complex
$\Gamma _\bullet M$ considered by Nguy{\~\ecircumflex}n H. V. H{\uhorn}ng  and  
Nguy{\~\ecircumflex}n Sum
\cite{hung_sum} (their arguments adapt to characteristic two).
\end{proof} 

Recall that the category of chain complexes for an abelian category is abelian.

\begin{cor}
 There is an exact functor $\dcx : \amod \rightarrow \chcx (\amod)$ with values
in 
$\nat$-graded chain complexes defined on an $\cala$-module $M$  by 
\begin{eqnarray*}
 \dcx_n M &: =& \Sigma R_s (\Sigma^{s-1} M) \\
d_n : \dcx_n M \rightarrow \dcx_{n-1} M &:=&  \Sigma d_{s,\Sigma^{s-1}M}.
\end{eqnarray*}
\end{cor}

\begin{proof}
 Proposition \ref{prop:d_squared} implies that $\dcx M$ is a chain complex and
the construction is functorial. 
 Since $R_s : \amod \rightarrow \amod$ is an exact functor (forgetting the
action of $D(s)$) and $\Sigma$ is exact, the functor $\dcx$ is exact.
\end{proof}

As shown by the work of Singer on the derived functors of iterated loop functors
\cite{singer_loops_II}, a key input to the proof of the main result is to have 
a 
short exact 
sequence of
complexes which gives rise to the long exact sequence of derived functors of
destabilization. 

\begin{nota}
For $M$ an $\cala$-module and $s \in \nat$, let $\dhom_s M$ denote $H_s
(\dcx_\bullet M)$, so that $\dhom_0 M= D M$, by 
Proposition \ref{prop:coker_SigmaD}.
\end{nota}

\begin{prop}
\label{prop:destab_ses_ch_cx}
For $M$ an $\cala$-module, there is a natural short exact sequence of chain
complexes:
\[
 0
\rightarrow 
\Sigma^{-1} \dcx_\bullet (\Sigma M) 
\rightarrow  
\dcx_\bullet M 
\rightarrow 
\Sigma^{-1} \Phi \dcx_{\bullet -1} (\Sigma M) 
\rightarrow 
0.
\]
Moreover, in homology this induces a long exact sequence in $\amod$:
\[
 \ldots 
\rightarrow 
\Sigma^{-1} \dhom_s (\Sigma M) 
\rightarrow 
\dhom _s M 
\rightarrow 
\Sigma^{-1}
\Phi \dhom_{s-1}(\Sigma M)  
\stackrel{\lambda_{s-1}}{\longrightarrow} 
\Sigma^{-1} \dhom _{s-1}(\Sigma M)
\rightarrow 
\ldots .
\]
The connecting morphism $\lambda_0$ identifies with $\Sigma^{-1} \lambda _{DM}$,
using the identification $\dhom _0 M = DM$.
\end{prop}

\begin{proof} 
Indications. (Cf. \cite{p_destab}, which treats odd characteristic.)
 The first statement follows from the naturality of the construction of the
chain complex of Proposition \ref{prop:ses_Rs_amod} 
and of the differential. 

For the final statement, the long exact sequence is the long exact sequence in
homology, using the exactness of the functors 
$\Sigma$ and $\Phi$. The identification of $\lambda_0$ is  straightforward.
\end{proof}

\begin{lem}
 \label{lem:conn_dhom}
For $M$ an $\cala$-module and $s \in \nat$, 
$
 \conn (\dhom_s M) \geq 2^s (\conn M + s). 
$
\end{lem}

\begin{proof}
 Straightforward.
\end{proof}

\begin{prop}
 \label{prop:vanish_proj_gen}
$\dhom_s (\Sigma^t \cala ) =0$ $\forall t \in \zed$ and $ 0< s \in \nat$.
\end{prop}

\begin{proof}
 For $s=1$ and $t\in\zed$ recall that $D (\Sigma^{t+1} \cala)= F(t+1)$, so that
the
long exact sequence of Proposition \ref{prop:destab_ses_ch_cx}
 is of the following form:
\[
 \ldots 
\rightarrow 
\Sigma^{-1} \dhom _1 (\Sigma^{t+1} \cala) 
\stackrel{\alpha_{t+1}}{\rightarrow} 
\dhom_1 (\Sigma^t \cala) \rightarrow 
\Sigma^{-1}\Phi F(t+1) 
\stackrel{\Sigma^{-1}\lambda}{\rightarrow}
\Sigma^{-1}F (t+1) \rightarrow \ldots . 
\]
The morphism $\lambda$ is injective, since $F(t+1)$ is reduced, hence the
morphism $\alpha_{t+1}$ is surjective. 
 Since $\conn (\dhom_1 (\Sigma^{t+1} \cala)) \rightarrow \infty $ as $t
\rightarrow \infty$, by Lemma \ref{lem:conn_dhom}, it follows that $\dhom_1
(\Sigma^t \cala)=0$ $ \forall t
\in\zed$.

This forms the initial step of an induction upon $s$; the inductive step is
similar (but easier).
\end{proof}

\begin{thm}
\label{thm:deriv_destab}
For $M$ an $\cala$-module, there is a natural isomorphism 
$$H_s (\dcx M) \cong D_s M .
$$
\end{thm}

\begin{proof}
 This follows  by standard arguments
of homological algebra, since $\dhom _0 M = D M$ by Proposition
\ref{prop:coker_SigmaD} and $\dhom_s$ vanishes for $s>0$ on the projectives of
$\amod$, by Proposition \ref{prop:vanish_proj_gen}.
\end{proof}

\begin{rem}
From the construction, it is not clear  {\em a priori} that the homology of the
complex should be  unstable.
\end{rem}

From this result, one recovers immediately one of the main results of Lannes 
and Zarati
\cite{lannes_zarati_deriv_destab}:

\begin{cor}
\label{cor:deriv_destab_LZ}
 For $M$ an unstable module and $s \in \nat$, there is a natural isomorphism 
\[
 D_s (\Sigma^{1-s} M) \cong 
\Sigma R_s M 
\]
and a short exact sequence of unstable modules 
\[
0 
\rightarrow 
R_sM 
\rightarrow 
D_s (\Sigma^{-s}M) 
\rightarrow 
\Omega_1 D _{s-1} (\Sigma^{1-s}M) 
\rightarrow 
0.
\]
In particular, if $M$ is reduced, then 
$
 D_s(\Sigma^{-s} M) \cong R_s M.
$
\end{cor}

\begin{proof}
The first statement is a consequence of the vanishing of the relevant 
differentials in
the chain complex $\dcx_\bullet M$ under the given 
hypotheses and the second follows from the short exact sequence of Corollary
\ref{cor:destab_groth_ses}. 
Finally it is clear that $R_s M$ is reduced if $M$ is reduced.
Hence, by induction on $s$, one sees that 
$D_s (\Sigma^{-s}M )$ is reduced and the $\Omega_1$ term vanishes.
\end{proof}

\begin{rem}
 Kuhn and McCarty \cite{kuhn_mccarty} (who work with homology) give a geometric 
construction of the analogous 
chain complex, notably giving a geometric construction of the Singer functors 
and the differential. The reader should 
compare the above with their approach, which shows the relationship with the 
Dyer-Lashof operations.
\end{rem}

\begin{exer}
Show that, if $M$ is a finite $\cala$-module (of finite total 
dimension), then the derived functors $D_sM$ are all non-trivial
for $s\gg0$.
\end{exer}

\subsection{Iterated loops}
\label{subsect:iterated_loops}

Fix $ t \in \nat$, which corresponds to the number of loops $\Omega^t$.

\begin{nota}
For $t \in \nat$, 
 let $R_{1/t} : \amod \rightarrow \field[u]\dash\amod$ denote the  functor 
defined on an $\cala$-module $M$ by 
 \[
 R_{1/t} M := \field[u]/(u^t) \otimes_{\field[u]}R_1 M,
\]
equipped with the natural projection 
$
 R_1 M \twoheadrightarrow R_{1/t}M
$ in $\field[u]\dash\amod$.
\end{nota}

\begin{exam}
For $M$ an $\cala$-module, $R_{1/0}M=0$ and there is a natural isomorphism 
$R_{1/1} M \cong \Phi
M$.
\end{exam}

\begin{exer}
Show that 
\begin{enumerate}
 \item 
$R_{1/t}$ is exact;
\item
$R_{1/t}$ induces a functor $\unstalg \rightarrow \field[u]/(u^t) \downarrow
\unstalg.$
\end{enumerate}
\end{exer}

\begin{lem}
\label{lem:truncate_differential}
For $N$  an unstable module, the differential $d_{\Sigma^{-t}N}$ induces
a commutative diagram
\[
 \xymatrix{
R_1 (\Sigma^{-t}N)  \ar[r]^{d_{\Sigma^{-t}N}}
\ar@{->>}[d]
&
\Sigma^{-t-1}N
\ar@{=}[d]
\\
R_{1/t}
(\Sigma^{-t}N)  
\ar[r]_{d_{1/t,N}}
&
\Sigma^{-t-1}N,
}
\]
in particular the morphism $d_{1/t,N} : R_{1/t}(\Sigma^{-t}N)  \rightarrow
\Sigma^{-t-1}N$ is $\cala$-linear.
\end{lem}

\begin{proof}
 Straightforward.
\end{proof}

The following underlines that the instability hypothesis is essential here:

\begin{exer}
Give an example of an $\cala$-module $N$ (necessarily not unstable) for which 
the diagram is {\em not} 
commutative.
\end{exer}

\begin{prop}
\label{prop:coker_residue_loops}
For $N$ an unstable module,  the cokernel of $$\Sigma  d_{1/t,N} : \Sigma
R_{1/t}(\Sigma^{-t}N)  \rightarrow \Sigma^{-t}N$$
 is $\Omega ^t N$. 
\end{prop}

\begin{proof}
This is left as an important exercise for the reader. (Cf. Proposition 
\ref{prop:coker_SigmaD}.)
\end{proof}

\begin{defn}
 For  integers $0\leq s \leq t$, let $R_{s/t} : \amod \rightarrow
D(s)\dash\amod$ denote the functor defined on an $\cala$-module $M$ by 
\[
 R_{s/t}M := 
\mathrm{image}
\{
R_s M \hookrightarrow (R_1)^{\circ s} M 
\twoheadrightarrow 
R_{1/t} \circ R_{1/t-1} \circ \ldots \circ R_{1/t-s+1} M
\},
\]
equipped with the canonical surjection $
 R_s M \twoheadrightarrow R_{s/t}M
$ in $D(s)\dash\amod$.
\end{defn}

\begin{rem}
The functor 
 $R_{s/t}$ is zero if $s>t$, since $R_{1/0}=0$.
\end{rem}

\begin{exer}
(This is somewhat harder than some previous exercises.) 
For integers $0\leq s \leq t$ and an $\cala$-module $M$, describe $R_{1/t} 
\circ 
R_{1/t-1} \circ \ldots
\circ R_{1/t-s+1} M$ explicitly as a quotient of 
$(R_1)^{\circ s}M$ by specifying the ideal $I_{s, t} \subset (R_1)^{\circ s}
\field$  such that 
\[
R_{1/t} \circ R_{1/t-1} \circ \ldots \circ R_{1/t-s+1} M \cong 
\big (
(R_1)^{\circ s} \field 
\big)/I_{s,t}
\otimes_{(R_1)^{\circ s} \field}
(R_1)^{\circ s} M.
\]
Deduce from this an analogous description of $R_{s/t} M $ in terms of $R_sM$.
\end{exer}

\begin{prop}
\label{prop:properties_Rs/t}
For integers $1 \leq s \leq t$,
\begin{enumerate}
 \item 
$R_{s/t} : \amod \rightarrow D(s)\dash\amod$ is exact;
\item
$R_{s/t}$ restricts to an exact functor on unstable modules 
$R_{s/t} : \unst \rightarrow D(s)\dash\unst$.
\end{enumerate}
\end{prop}

\begin{proof}
 Straightforward.
\end{proof}

By construction, for an $\cala$-module $M$,  there is a natural inclusion 
$
 R_{s/t} M \hookrightarrow R_{s-1/t}  R_{1/t-s+1}M,
$
which fits into the commutative diagram:
\[
 \xymatrix{
R_s M
\ar@{->>}[d]
\ar@{^(->}[r]
&
R_{s-1}R_1 M
\ar@{->>}[d]
\\
R_{s/t}M 
\ar@{^(->}[r]
&
R_{s-1/t}R_{1/t-s+1}M.
}
\]

Hence, as in the construction of $d_{1/t,N}$, there is an induced morphism in
$\amod$, which is given for $N $ an unstable module by the composite:
\[
\xymatrix{
 R_{s/t}
(\Sigma^{-(t-s+1)}N) 
 \ar[r]
\ar[rd]_{d_{s/t, N} }
&
 R_{s-1/t}  R_{1/t-s+1}(\Sigma^{-(t-s+1)}N) 
\ar[d]^{R_{s-1/t}d_{1/t-s+1,N}}\\
&
R_{s-1/t}(\Sigma^{-(t-(s-1)+1))}N) 
.
}\]

\begin{lem}
 For  integers $1 \leq s \leq t$ and an unstable module $N$, the following
diagram commutes:
\[
 \xymatrix{
R_s (\Sigma^{-(t-s+1)}N)
\ar[rr]^{d_{s,\Sigma^{-(t-s+1)}N}}
\ar@{->>}[d]
&&
R_{s-1}(\Sigma^{-(t-(s-1)+1)}N)
\ar@{->>}[d]
\\
R_{s/t}(\Sigma^{-(t-s+1)}N)
\ar[rr]_{d_{s/t,N}}
&&
R_{s-1/t}(\Sigma^{-(t-(s-1)+1)}N).
}
\]
\end{lem}

\begin{proof}
 Straightforward. 
\end{proof}

\begin{rem}
The hypothesis that $N$ be unstable is essential for this
compatibility, as in Lemma \ref{lem:truncate_differential}.
\end{rem}

There is also an analogue of the short exact sequence of Proposition
\ref{prop:ses_Rs_amod}, based on the observation 
that, for $t\geq 1$ and $M$ an $\cala$-module, the natural surjection $\rho_1: 
R_1 M \twoheadrightarrow
\Phi M$ factorizes naturally across a surjection 
\[
 \rho_1 : R_{1/t} M \twoheadrightarrow \Phi M.
\]

\begin{prop}
\label{prop:loop_ses_Rs/t}
 For integers $1\leq s \leq t$, the morphism $\rho_1$ induces a short exact
sequence for the functors $ R_{*/t}$ which 
forms the bottom row of the commutative diagram for $M$ an $\cala$-module
\[
 \xymatrix{
0
\ar[r]&
\Sigma^{-1} R_s \Sigma M 
\ar[r]
\ar@{->>}[d]
&
R_s M
\ar[r]^{\rho_s}
\ar@{->>}[d]
&
\Phi R_{s-1}M
\ar[r]
\ar@{->>}[d]
&
0
\\
0\ar[r]
&
\Sigma^{-1} R_{s/t-1} \Sigma M 
\ar[r]
&
R_{s/t} M 
\ar[r]_{\rho_s}
&
\Phi R_{s-1/t-1} M
\ar[r]
&
0,
}
\]
where the top row is provided by Proposition \ref{prop:ses_Rs_amod} and the
vertical morphisms are the canonical surjections.
\end{prop}

\begin{proof}
(Indications.)
 The only non-trivial point is to identify the kernel in the bottom row; this is
clear in the case $s=1$ and the higher cases
are treated by induction.  
\end{proof}

\begin{exer}
For $M$  a finite $\cala$-module (i.e. the total dimension is finite) and 
integers $1 \leq s \leq t$, 
\begin{enumerate}
 \item 
show that the total dimension of $R_{s/t}M$ is finite;
\item
calculate $\conn (R_{s/t}M)$ in terms of $\conn M $ and $s,t$;
\item
calculate the top dimension of $R_{s/t}M$ in terms of the top dimension of $M$.
\end{enumerate}
(Hint: use information on the algebra $R_{s/t}\field$, which can be obtained 
inductively using Proposition \ref{prop:loop_ses_Rs/t}.)
\end{exer}

\begin{defn}
 Let $\lcx^t _\bullet : \unst \rightarrow \chcx (\amod)$ denote the exact 
functor
defined on an unstable module $N$ 
by  $\lcx^t _s N:= \Sigma R_{s/t}(\Sigma^{-(t-s+1)}N)$ and with differential
$\Sigma d_{s/t,N}$. 
\end{defn}

\begin{rem}
 The fact that $\lcx^t$ is a chain complex (namely $d^2=0$) is a {\em 
consequence} of the corresponding result for  $\dcx$, which
follows from Proposition \ref{prop:d_squared}.
\end{rem}

\begin{exer}
Show that the chain complex $\lcx^t_\bullet N$ is bounded for any
unstable module, $N$. Namely,  $\lcx ^t_s N = 0$ for $s >t$.
\end{exer}

\begin{prop}
\label{prop:compatibility_dcx_lcx}
For $N$ an unstable module, there is a natural surjection of chain complexes:
\[
 \dcx (\Sigma^{-t}N) 
\twoheadrightarrow 
\lcx^t (N).
\]
Moreover, the short exact sequences of Proposition \ref{prop:loop_ses_Rs/t}
induce a short exact sequence of chain complexes
which fits into the commutative diagram
\[
 \xymatrix{
0
\ar[r]
&
\Sigma^{-1}\dcx _\bullet (\Sigma^{-t+1}N)
\ar[r]
\ar@{->>}[d] 
&
\dcx _\bullet (\Sigma^{-t}N)
\ar[r]
\ar@{->>}[d]
&
\Sigma^{-1}\Phi \dcx_{\bullet -1} (\Sigma^{-t+1}N)
\ar@{->>}[d]
\ar[r]
&
0
\\
0
\ar[r]
&
\Sigma^{-1}\lcx^{t-1} _\bullet (N)
\ar[r]
&
\lcx^t _\bullet (N)
\ar[r]
&
\Sigma^{-1} \Phi \lcx^{t-1}_{\bullet -1} (N)
\ar[r]
&
0.
}
\]
\end{prop}

\begin{exer}
Show that  $\lcx ^1_\bullet N = \big(\Sigma^{-1} \Phi N
\rightarrow \Sigma^{-1} N\big)$.
\end{exer}

\begin{rem}
 The functor $N \mapsto \lcx^t_\bullet N$, for $N$ an unstable module, has the 
same formal properties as that
constructed by Singer in \cite{singer_loops_II}. The current presentation, 
being 
based upon quotients of the Singer functors and 
the Singer differential, makes explicit the relationship between $\dcx 
\Sigma^{-t}$ and $\lcx^t$. 
\end{rem}

\begin{thm}
 For $N$ an unstable module and $s,t\in \nat$, there is a natural isomorphism 
\[
 \Omega^t_s N \cong H_s (\lcx^t _\bullet N). 
\]
Moreover, the surjection of chain complexes $\dcx_\bullet (\Sigma^{-t}N) 
\twoheadrightarrow \lcx^t_\bullet (N)$ of Proposition 
\ref{prop:compatibility_dcx_lcx} induces the 
natural transformations $D_s (\Sigma^{-t}N) 
\rightarrow 
\Omega^t_s N$ in homology.
\end{thm}

\begin{proof} (Indications.)
 The proof of the first point is formally similar to that of Theorem 
\ref{thm:deriv_destab} but the
inductive step is easier, since  a double induction 
on $t$ and $s$ can be used. This argument is identical to that used in Singer
\cite{singer_loops_II}, which only requires
the formal properties of the chain complex. 
\end{proof}

The following result is analogous to Corollary \ref{cor:deriv_destab_LZ} (and is
implicit in \cite{singer_loops_II}).

\begin{cor}
 For $N$ an unstable module, $s,t \in \nat$  and $k \in \nat$ such that  $k \geq
t-s+1$, there is a natural isomorphism:
\[
 \Omega^t_s (\Sigma^{k} N) 
\cong 
\Sigma R_{s/t} (\Sigma^{k-(t-s+1)} N)
\]
of unstable modules.

Under these hypotheses, there is a short exact sequence of unstable modules:
\[
 0
\rightarrow 
R_{s/t-1} (\Sigma^{k-(t-s+1)} N) 
\rightarrow 
\Omega^t_s (\Sigma^{k-1} N) 
\rightarrow 
\Omega_1 \Omega^{t-1}_{s-1} (\Sigma^{k-1} N) 
\rightarrow 
0.
\]
\end{cor}

\begin{proof}
 Straightforward.
\end{proof}

\begin{rem}
Unlike the functor $R_1$, the functor $R_{1/t}$ restricted to $\unst$ does not 
send reduced unstable modules to reduced objects if $t>1$; in particular, 
$R_{1/t} \field \cong \field[u]/{u^t}$ is not reduced for $t>1$. 
\end{rem}

\begin{exam}
 For $t=1$ and $N$ an  unstable module, 
\begin{enumerate}
 \item 
for $s=0$ and $k \geq 2$, $\Omega (\Sigma^k N) \cong \Sigma \Sigma^{k-2}
N \cong \Sigma ^{k-1}N$, as expected;
\item
for $s=1$, we require $k \geq 1$ and get $\Omega_1 (\Sigma^k N) \cong \Sigma
R_{1/1} (\Sigma^{k-1}N) \cong \Sigma^{-1} \Phi \Sigma^k N$, using the
identification $R_{1/1} \cong \Phi$ and $\Phi \Sigma\cong \Sigma^2 \Phi$. 
\end{enumerate}
\end{exam}

\subsection{The Lannes-Zarati homomorphism}

The derived functors of destabilization 
are related to  homology over the Steenrod algebra as follows.  For $N$ an 
$\cala$-module, there is a natural transformation $D N  \to \field 
\otimes_{\cala} N$ 
 of functors from $\cala$-modules to $\cala$-modules, where $ \field 
\otimes_{\cala} N$ is given the trivial $\cala$-module structure. This is obtained 
by applying the destabilization functor $D$ to the quotient 
$N \twoheadrightarrow \field \otimes_{\cala} N$ and then composing with the 
canonical  
inclusion:
\[
D N  \twoheadrightarrow D(\field \otimes_{\cala} N) =  (\field \otimes_{\cala} 
N)^{\geq 0} 
\stackrel{\subset}{\longrightarrow} 
 \field \otimes_{\cala} N.     
\]
This  passes to  derived functors to give    
$$
 D_s N \rightarrow \tor_s^{\cala} (\field, N).
$$

Now, as in Remark \ref{rem:alpha_M_s}, for $M$ an unstable module, there is a 
natural transformation
$
 \alpha^M_s : D_s (\Sigma^{-s} M) \rightarrow R_s M. 
 $ 
Moreover,  by \cite[Théorème 2.5]{lannes_zarati_deriv_destab},  $\alpha^{\Sigma 
M}_s$ induces an isomorphism 
\[
 D_s (\Sigma ^{1-s } M)
 \stackrel{\cong}{\longrightarrow} 
 \Sigma R_s M .
\]

Hence there is a natural morphism of $\cala$-modules 
$\Sigma  R_s M \rightarrow \tor_s^\cala (\field, 
\Sigma^{1-s} M) \cong \Sigma \tor_s^\cala (\field, 
\Sigma^{-s} M) $ and thus 
$
 \field \otimes_\cala R_s M 
 \rightarrow 
 \tor_s^\cala (\field, 
\Sigma^{-s} M).
$ 

The dual of this map, 
\begin{eqnarray*}
\ext_{\cala}^{s}(\Sigma^{-s} M, \field) \to (\field \otimes_\cala R_s 
M)^*, 
\end{eqnarray*}
is the Lannes-Zarati homomorphism. 

\begin{rem}
\label{rem:LZ_Hurewicz}
The Lannes-Zarati homomorphism corresponds to an associated graded of the 
mod $2$ Hurewicz map 
 \begin{eqnarray*}
 \pi_* (\Omega^\infty  
\Sigma^\infty X) 
\rightarrow 
H_* (\Omega^\infty \Sigma^\infty X)   
\end{eqnarray*}
when $M$ is the reduced cohomology of the pointed space $X$. (The proof of this 
assertion is  sketched in \cite{Lannes} and  \cite{Goerss}.)
\end{rem}

Singer \cite{singer_invt_lambda} constructed a chain complex $\Gamma^+_\bullet
M$ that computes the homology of $M$ over the Steenrod algebra 
as a sub-complex of a larger complex $\Gamma_\bullet M$. Using the method of
Nguy{\~\ecircumflex}n H. V. H{\uhorn}ng  and  Nguy{\~\ecircumflex}n Sum 
\cite{hung_sum} (adapted to the prime $2$) and the material 
presented here, 
it is possible to show the following:

\begin{prop}
 There is a natural inclusion of chain complexes 
\[
 \dcx _\bullet M \hookrightarrow \Gamma^+_\bullet M
\]
that induces the dual Lannes-Zarati homomorphism in homology.
\end{prop}
 
\begin{rem}
 Alternative approaches to chain level representations of the Lannes-Zarati 
homomorphism have been given (see \cite{HT} for example). 
\end{rem}

\section{Perspectives}

This section indicates some recent developments and open problems\footnote{This 
material was not presented in the original lectures.}. The ground field  is 
sometimes taken to be $\field_p$ with $p$ odd.


\subsection{The spherical class conjecture and related problems}  For $X$ a 
pointed space, the 
mod $2$ Hurewicz map induces 
\[
h :  \pi_* (Q X ) 
 \rightarrow 
H_* (QX ; \field_2),  
 \]
where $QX := \Omega^\infty \Sigma^\infty X$ is the associated infinite loop 
space. Here, $\pi_* (QX)$ identifies with the stable homotopy groups $\pi^S_* 
(X)$ of $X$. 

The famous Curtis conjecture asserts the following:

\begin{conj}
 \cite{Curtis}
 For $X = S^0$,  the kernel contains all classes except those of odd Hopf or 
Kervaire invariant.  
\end{conj}

Various generalizations of the Curtis conjecture have been proposed. Many 
involve the Adams filtration, which we recall here for 
a generalized homology theory $E_*$:

\begin{defn}
\label{defn:Adams_filtration}
 A map $f : X \rightarrow Y$ between spectra has $E$-Adams filtration at least $s 
\in \nat$ if there is a factorization 
 \[
 \xymatrix{
 X = X_0
 \ar[r]^{f_0} 
 \ar@/_1pc/[rrrr]_f
 &
 X_1 
 \ar[r]^{f_1}
 &
 \ldots 
 \ar[r]
 &
 X_{s-1} 
 \ar[r]^{f_{s-1}}
 &
 X_s = Y
 }
 \]
 where, for each $i$, $E\wedge f_i $ is null.
\end{defn}

\begin{rem}
 For ordinary cohomology $H\field_p$, this corresponds to the filtration arising 
from the Adams spectral sequence, 
 which has $E_2$-page 
 $$\ext^{s,t}_\cala (H^* (Y), H^* (X)),$$
 where the cohomological degree $s$ corresponds to the Adams filtration. 
\end{rem}

Nguy{\~\ecircumflex}n H. V. H{\uhorn}ng has proposed the following generalization of the Curtis conjecture 
(see \cite{HT}):

\begin{conj}
[The generalized spherical class conjecture]
\label{conj:gen_spher_class}
For $X$ a pointed space, the mod $2$ Hurewicz homomorphism 
$$h:  
\pi_{*}(Q X) \to H_{*}(Q X; \field_2)$$
 vanishes on classes of Adams filtration greater than $2$.
\end{conj}

The Hurewicz map can also be studied in the stable context. Taking $Y$ to be a 
spectrum and using $\field_p$ coefficients, 
we have the mod $p$ Hurewicz map:
\[
 h : 
 \pi_* (Y) 
 \rightarrow 
 H_* (\Omega^\infty Y ; \field_p).
\]

Motivated by his recent work on the Hurewicz map relating the Adams filtration 
to a certain {\em augmentation ideal filtration} (see Section 
\ref{subsect:gen_LZ} below),  Kuhn  \cite{K_adams_filt} has 
proposed:

\begin{conj}
\label{conj:Kuhn_Curtis}
Let $Y$ be a spectrum such that $H^* (Y; \field_p)$ is finitely generated as 
an $\cala$-module. Then there exists $s$ such that the kernel of 
$$h : 
 \pi_* (Y) 
 \rightarrow 
 H_* (\Omega^\infty Y ; \field_p)$$
 contains all elements of Adams filtration at 
least $s$. 
\end{conj}

\begin{rem}
 If $Y = \Sigma^\infty X$, for $X$ a pointed space (satisfying certain finiteness hypotheses), 
 Gaudens and Schwartz \cite{GS} have shown that the hypothesis of Conjecture \ref{conj:Kuhn_Curtis} 
 that $H^* (Y; \field_p)$ is finitely generated as an $\cala$-module implies that $H^* (Y ; \field_p)$ is actually 
 finite. 
 
 Hence, in the unstable realm, Conjectures \ref{conj:gen_spher_class} and \ref{conj:Kuhn_Curtis} should be considered on finite, pointed CW complexes. 
 Here, the former asserts the stronger form that $s$ can be taken to be $2$. It is possible that the more general stable conjecture proposed by Kuhn could shed further light
  on the unstable case. 
\end{rem}

\begin{rem}
The above conjectures are hard. As a first step, it is interesting to consider related 
algebraic conjectures, as below.
\end{rem}

 Recall that the Lannes-Zarati homomorphism is an algebraic approximation to the 
Hurewicz morphism (see Remark \ref{rem:LZ_Hurewicz}). The following was proposed by 
Nguy{\~\ecircumflex}n H. V. H{\uhorn}ng (see \cite{HT}):

\begin{conj}
[The generalized algebraic spherical class conjecture]
\label{conj:gen_alg_spher}
For $M$ an unstable module, the mod $2$ Lannes-Zarati homomorphism 
$$
\ext_{\cala}^{s}(\Sigma^{-s} M, \field) \to (\field_2 \otimes_\cala R_s 
M)^* 
$$
vanishes in  positive degree for $s>2$.
\end{conj}

\begin{rem}
This conjecture  has been proved for $M=\field_2$ and  $s\in \{3,4, 5\}$ 
(see \cite{H_spherical,Hung1999,Hung2003,HQT}).  However, the 
general case seems to be beyond reach using existing techniques.
 \end{rem}
 
A further simplification is obtained by restricting to the image of the  
algebraic Singer 
transfer \cite{singer_transfer}. This leads to the following (cf. 
\cite[Conjecture 1.6]{HT} and the 
presentation in \cite{Hung_Powell}), in which $R_s M$ is considered as a 
submodule of $P_s \otimes M$, where 
$P_s := H^* (BV_s)$, $V_s$ a rank $s$ elementary abelian $2$-group.

\begin{conj} 
[The weak generalized algebraic spherical class conjecture] 
\label{conj:weak_alg}
Let $M$ be an unstable $\cala$-module (over $\field_2$) and $s>2$ be an integer. 
Then every  positive degree element of 
the Singer construction $R_sM$ is $\cala$-decomposable in 
$P_s\otimes M$. 
\end{conj}

The full weak generalized algebraic spherical class conjecture was proved as 
the main result of \cite{Hung_Powell}:

\begin{thm}
\label{thm:weak_algebraic}
 For $M$ an unstable module (over $\field_2$) and  $2<s \in \nat$, the morphism  
\[
R_s M \rightarrow \field_2 \otimes_{\cala} (P_s \otimes M) 
\]
is trivial  on elements of positive degree.
\end{thm}

\begin{rem}
\ 
\begin{enumerate}
  \item
 Theorem \ref{thm:weak_algebraic} gives evidence supporting Conjecture \ref{conj:gen_alg_spher}, in particular providing a result valid for 
all unstable modules $M$. Theorem \ref{thm:weak_algebraic} may also lead to further progress on the Conjecture;  this is a subject for future research. 
  \item 
  Current approaches to Conjecture \ref{conj:gen_alg_spher} rely heavily upon knowledge of the 
structure of $\ext_{\cala}^{s}(\Sigma^{-s} M, \field)$, hence are difficult to 
generalize. One motivation for studying 
  chain-level constructions of the Lannes-Zarati map is to 
develop methods which do not depend upon such information. 
 \end{enumerate}
\end{rem}

\begin{rem}
\ 
\begin{enumerate}
 \item 
Implicit in the above  is the relationship between the Singer functors 
$R_s$ and the calculation of the (co)homology of the infinite loop space $QX$ 
associated to a pointed space $X$. 
  This relationship is clearer when working in {\em homology}, where the 
Dyer-Lashof operations appear naturally. (The reader is also referred to  
\cite[Section 2]{K_whitehead}, where a presentation is given 
  which makes use of the Hecke algebra.)
 \item 
The Singer functors appear naturally in other problems. For example, in 
\cite{K_whitehead}, Kuhn shows how unstable module theory in conjunction with an 
understanding of the Singer functors leads to a conceptual, modern proof of the 
Whitehead Conjecture.
\end{enumerate}
\end{rem}

\subsection{Generalizations of the Lannes-Zarati homomorphism}
\label{subsect:gen_LZ}

Kuhn \cite{K_adams_filt} has recently provided a new perspective which may 
allow 
Lannes-Zarati theory to be 
extended to certain generalized cohomology theories. Kuhn's constructions rely 
on working with {\em highly structured commutative ring spectra} (here taken to be commutative 
$S$-algebras, in the sense of \cite{EKMM}). 

\begin{hyp}
\label{hyp:E_nice}
 Let $E$ be a commutative $S$-algebra such that $E$ is connective and the unit 
map induces a surjection $\pi_0 (S) \rightarrow \pi_0 (E)$. 
\end{hyp}

\begin{exam}
The ring spectrum $E$ can be taken to be the mod $p$ Eilenberg-MacLane spectrum 
$H\field_p$, for any prime $p$.
\end{exam}

For $X$ a spectrum, the $E$-based unstable Hurewicz map $\pi_* (X) \rightarrow 
E_* (\Omega^\infty X) $ is induced by the map of spaces: 
\[
 h : \Omega^\infty X 
 \rightarrow 
 \Omega^\infty 
 (E \wedge \Sigma^\infty \Omega^\infty X)
\]
which is adjoint to the map $\Sigma^\infty \Omega^\infty X \rightarrow E\wedge \Sigma^\infty \Omega^\infty X$
 induced by the unit of $E$.

To state Kuhn's result, recall that the $E$-based Adams resolution of $X$ is the natural decreasing 
filtration 
\[
 \ldots \rightarrow X(2) \rightarrow X(1) \rightarrow X(0) = X,
\]
where $X(s+1)$ is the fibre of $X(s) \rightarrow E \wedge X(s)$ induced by the 
unit map of $E$.

\begin{rem}
 The $E$-based Adams resolution gives another viewpoint on the $E$-Adams 
filtration introduced in Definition 
 \ref{defn:Adams_filtration}. 
 
 To illustrate this, observe that the map $X(s) \rightarrow E \wedge 
X(s)$ admits a retract after smashing with $E$ (using the multiplicative 
structure of $E$),   hence $$E\wedge \big(X(s+1) \rightarrow X(s)\big)$$ is null. In particular,  any map of spectra $U 
\rightarrow X$ that factors  across $X(s) \rightarrow X$ has Adams filtration at least $s$.
\end{rem}

As explained in \cite{K_adams_filt}, $\Sigma^\infty \Omega ^\infty X$ can be 
given a natural, non-unital 
commutative $S$-algebra structure and admits a decreasing 
{\em augmentation ideal filtration}: 
\[
 \ldots 
 \rightarrow 
 I^3 (X) \rightarrow 
I^2 (X) 
\rightarrow 
I (X) \simeq \Sigma^\infty \Omega ^\infty X. 
 \]

The fundamental new input from \cite{K_adams_filt} is then the following 

\begin{thm}
\label{thm:Kuhn_compare_filt}
Suppose that $E$ satisfies Hypothesis \ref{hyp:E_nice} and let $p$ be a prime. 
Then, localized away from $(p-1)!$, the Hurewicz map lifts to a 
map of towers:
\[
 h_s : \Omega^\infty X(s) 
 \rightarrow 
 \Omega^\infty (E \wedge I^{p^s} (X))
\]
that relates the $E$-based Adams filtration with the augmentation ideal 
filtration.
\end{thm}

\begin{rem}
 If $X$ is a connective spectrum, there are fibration sequences of spectra for $t \in \nat$:
 \[
  I^{t+1} (X) \rightarrow I^t (X) \rightarrow D_t X, 
 \]
where $D_t X:= X^{\wedge t} _{h \mathfrak{S}_t}$ is the $t$th extended power 
construction on the spectrum $X$.

In the case $E=H\field_p$, the Singer functors $R_s$ (or, rather, their homological 
counterparts) are closely related to the
 calculation of the  homology $E_* (D_{p^s} X)$ in terms of $E_* (X)$ (see the 
construction of the homological Singer functors in 
\cite{kuhn_mccarty,K_whitehead} and also 
\cite{K_adams_filt}). 
 This establishes the relationship of the above with Lannes-Zarati theory.
\end{rem}

Kuhn \cite{K_adams_filt} observes that Lannes-Zarati theory can be generalized 
when the augmentation 
ideal filtration is known to split after smashing with $E$. This occurs for 
example for
\begin{enumerate}
 \item 
 $X$ the suspension spectrum $\Sigma^\infty Z$ of a space $Z$;
 \item 
 $E$ the $n$th Morava $E$-theory at the prime $p$, after localizing  the filtration  with 
respect to $n$th Morava $K$-theory.
\end{enumerate}

\begin{exam}
\cite{K_adams_filt} 
 Take $E = H \field_p$ and suppose that $Z$ is the suspension of a pointed 
space. In this case, 
the homological Singer functors appear  as the primitives of the Hopf algebra 
$H_* (Q Z; \field_p)$, and the mod $p$ 
Hurewicz map induces: 
 \[
  \pi^S_* (Z) 
  \rightarrow 
  \mathcal{R} H_* (Z ; \field_p) := \bigoplus R_s H_* (Z; \field_p)
 \]
(using {\em homological} Singer functors). This is filtration-preserving and 
recovers Lannes-Zarati's higher Hopf invariants at {\em all} primes 
\cite{lannes_zarati_hopf,LZhh}. 
\end{exam}

\begin{exam}
Let $E$ be the $n$th Morava $E$-theory at the prime $p$. For the technical details 
involved in applying Theorem \ref{thm:Kuhn_compare_filt}, the reader is referred to 
\cite{K_adams_filt} and, in particular, \cite[Corollary 1.17]{K_adams_filt}. 
 
It is expected that, when $X$ is a spectrum with $E_* (X)$ a finitely-generated 
free $E_*$-module, that there will be an {\em algebraic} $E$-theory Lannes-Zarati morphism.
 \end{exam}

 This provides a tantalizing glimpse of higher chromatic analogues 
of the theory outlined in these notes for singular cohomology, and represents 
a rich field for future research.


\begin{thebibliography}{EKMM97}

\bibitem[Cur75]{Curtis}
Edward~B. Curtis, \emph{The {D}yer-{L}ashof algebra and the {$\Lambda
  $}-algebra}, Illinois J. Math. \textbf{19} (1975), 231--246. \MR{0377885}

\bibitem[EKMM97]{EKMM}
A.~D. Elmendorf, I.~Kriz, M.~A. Mandell, and J.~P. May, \emph{Rings, modules,
  and algebras in stable homotopy theory}, Mathematical Surveys and Monographs,
  vol.~47, American Mathematical Society, Providence, RI, 1997, With an
  appendix by M. Cole. \MR{1417719}

\bibitem[Goe86]{Goerss}
Paul~G. Goerss, \emph{Unstable projectives and stable {${\rm Ext}$}: with
  applications}, Proc. London Math. Soc. (3) \textbf{53} (1986), no.~3,
  539--561. \MR{868458 (88d:55011)}

\bibitem[GS13]{GS}
G\'erald Gaudens and Lionel Schwartz, \emph{Applications depuis {$K(\Bbb
  Z/p,2)$} et une conjecture de {N}. {K}uhn}, Ann. Inst. Fourier (Grenoble)
  \textbf{63} (2013), no.~2, 763--772. \MR{3112848}

\bibitem[HM89]{harper_miller}
J.~R. Harper and H.~R. Miller, \emph{Looping {M}assey-{P}eterson towers},
  Advances in homotopy theory ({C}ortona, 1988), London Math. Soc. Lecture Note
  Ser., vol. 139, Cambridge Univ. Press, Cambridge, 1989, pp.~69--86.
  \MR{1055869 (91c:55032)}

\bibitem[HM16]{haugseng_miller}
Rune Haugseng and Haynes Miller, \emph{On a spectral sequence for the
  cohomology of infinite loop spaces}, Algebr. Geom. Topol. \textbf{16} (2016),
  no.~5, 2911--2947. \MR{3572354}

\bibitem[H{\uhorn}n97]{H_spherical}
Nguy{\~\ecircumflex}n H.~V. H{\uhorn}ng, \emph{Spherical classes and the
  algebraic transfer}, Trans. Amer. Math. Soc. \textbf{349} (1997), no.~10,
  3893--3910. \MR{1433119 (98e:55020)}

\bibitem[H{\uhorn}n99]{Hung1999}
\bysame, \emph{The weak conjecture on spherical classes}, Math. Z. \textbf{231}
  (1999), no.~4, 727--743. \MR{1709493}

\bibitem[H{\uhorn}n03]{Hung2003}
\bysame, \emph{On triviality of {D}ickson invariants in the homology of the
  {S}teenrod algebra}, Math. Proc. Cambridge Philos. Soc. \textbf{134} (2003),
  no.~1, 103--113. \MR{1937796}

\bibitem[HP16]{Hung_Powell}
Nguy{\~\ecircumflex}n H.~V. H{\uhorn}ng and Geoffrey {Powell}, \emph{{The
  A-decomposability of the Singer construction}}, arXiv:1606.09443 (2016).

\bibitem[HQT14]{HQT}
Nguy{\~\ecircumflex}n H.~V. H{\uhorn}ng, V\~o T.~N. Qu\`ynh, and
  Ng{\ocircumflex}~A. Tu{\'\acircumflex}n, \emph{On the vanishing of the
  {L}annes-{Z}arati homomorphism}, C. R. Math. Acad. Sci. Paris \textbf{352}
  (2014), no.~3, 251--254. \MR{3167575}

\bibitem[HS95]{hung_sum}
Nguy{\~\ecircumflex}n H.~V. H{\uhorn}ng and Nguy{\^e}n Sum, \emph{On {S}inger's
  invariant-theoretic description of the lambda algebra: a mod {$p$} analogue},
  J. Pure Appl. Algebra \textbf{99} (1995), no.~3, 297--329. \MR{1332903
  (96c:55024)}

\bibitem[HT15]{HT}
Nguy{\~\ecircumflex}n H.~V. H{\uhorn}ng and Ng{\ocircumflex}~A.
  Tu{\'\acircumflex}n, \emph{The generalized algebraic conjecture on spherical
  classes}, preprint 1564 ftp://file.viasm.org/Web/TienAnPham-15/, 2015.

\bibitem[KM13]{kuhn_mccarty}
Nicholas~J. Kuhn and Jason McCarty, \emph{The mod 2 homology of infinite
  loopspaces}, Algebr. Geom. Topol. \textbf{13} (2013), no.~2, 687--745.
  \MR{3044591}

\bibitem[Kuh14]{K_adams_filt}
Nicholas~J. Kuhn, \emph{{Adams filtration and generalized Hurewicz maps for
  infinite loopspaces}}, arXiv:1403.7501 (2014).

\bibitem[Kuh15]{K_whitehead}
\bysame, \emph{The {W}hitehead conjecture, the tower of {$S^1$} conjecture, and
  {H}ecke algebras of type {A}}, J. Topol. \textbf{8} (2015), no.~1, 118--146.
  \MR{3335250}

\bibitem[Lan88]{Lannes}
J.~Lannes, \emph{Sur le {$n$}-dual du {$n$}-\`eme spectre de {B}rown-{G}itler},
  Math. Z. \textbf{199} (1988), no.~1, 29--42. \MR{954749}

\bibitem[Lan92]{la}
Jean Lannes, \emph{Sur les espaces fonctionnels dont la source est le
  classifiant d'un {$p$}-groupe ab\'elien \'el\'ementaire}, Inst. Hautes
  \'Etudes Sci. Publ. Math. (1992), no.~75, 135--244, With an appendix by
  Michel Zisman. \MR{1179079 (93j:55019)}

\bibitem[LZ83]{LZhh}
Jean Lannes and Sa{\"{\i}}d Zarati, \emph{Invariants de {H}opf d'ordre
  sup\'erieur et suite spectrale d'{A}dams}, C. R. Acad. Sci. Paris S\'er. I
  Math. \textbf{296} (1983), no.~15, 695--698. \MR{705694 (85a:55009)}

\bibitem[LZ84]{lannes_zarati_hopf}
\bysame, \emph{{Invariants de Hopf d'ordre supérieur et suite spectrale
  d'Adams}}, Preprint, 1984.

\bibitem[LZ87]{lannes_zarati_deriv_destab}
\bysame, \emph{Sur les foncteurs d\'eriv\'es de la d\'estabilisation}, Math. Z.
  \textbf{194} (1987), no.~1, 25--59. \MR{MR871217 (88j:55014)}

\bibitem[Mar83]{margolis}
H.~R. Margolis, \emph{Spectra and the {S}teenrod algebra}, North-Holland
  Mathematical Library, vol.~29, North-Holland Publishing Co., Amsterdam, 1983,
  Modules over the Steenrod algebra and the stable homotopy category.
  \MR{738973 (86j:55001)}

\bibitem[MM65]{milnor_moore}
John~W. Milnor and John~C. Moore, \emph{On the structure of {H}opf algebras},
  Ann. of Math. (2) \textbf{81} (1965), 211--264. \MR{0174052 (30 \#4259)}

\bibitem[M{\`u}i75]{mui_mod_invt_symm}
Hu{\`y}nh M{\`u}i, \emph{Modular invariant theory and cohomology algebras of
  symmetric groups}, J. Fac. Sci. Univ. Tokyo Sect. IA Math. \textbf{22}
  (1975), no.~3, 319--369. \MR{0422451 (54 \#10440)}

\bibitem[M{\`u}i86]{mui_cohom_operations}
\bysame, \emph{Cohomology operations derived from modular invariants}, Math. Z.
  \textbf{193} (1986), no.~1, 151--163. \MR{852916 (88e:55015)}

\bibitem[Pow10]{powell_mod_loop}
Geoffrey M.~L. Powell, \emph{Module structures and the derived functors of
  iterated loop functors on unstable modules over the {S}teenrod algebra}, J.
  Pure Appl. Algebra \textbf{214} (2010), no.~8, 1435--1449. \MR{2593673}

\bibitem[Pow12]{p_singer_Rs}
\bysame, \emph{On unstable modules over the {D}ickson algebras, the {S}inger
  functors ${R}_s$ and the functors $\mathrm{{F}ix}_s$}, Algebr. Geom. Topol.
  \textbf{12} (2012), 2451--2491 (electronic).

\bibitem[Pow14]{p_destab}
\bysame, \emph{On the derived functors of destabilization at odd primes}, Acta
  Math. Vietnam. \textbf{39} (2014), no.~2, 205--236. \MR{3212661}

\bibitem[Pri70]{priddy}
Stewart~B. Priddy, \emph{Koszul resolutions}, Trans. Amer. Math. Soc.
  \textbf{152} (1970), 39--60. \MR{0265437 (42 \#346)}

\bibitem[Sch94]{schwartz_book}
Lionel Schwartz, \emph{Unstable modules over the {S}teenrod algebra and
  {S}ullivan's fixed point set conjecture}, Chicago Lectures in Mathematics,
  University of Chicago Press, Chicago, IL, 1994. \MR{MR1282727 (95d:55017)}

\bibitem[Sin80]{singer_loops_II}
William~M. Singer, \emph{Iterated loop functors and the homology of the
  {S}teenrod algebra. {II}. {A} chain complex for {$\Omega \sp{k}\sb{s}M$}}, J.
  Pure Appl. Algebra \textbf{16} (1980), no.~1, 85--97. \MR{MR549706
  (81b:55040)}

\bibitem[Sin81]{singer_new_chain_cx}
\bysame, \emph{A new chain complex for the homology of the {S}teenrod algebra},
  Math. Proc. Cambridge Philos. Soc. \textbf{90} (1981), no.~2, 279--292.
  \MR{MR620738 (82k:55018)}

\bibitem[Sin83]{singer_invt_lambda}
\bysame, \emph{Invariant theory and the lambda algebra}, Trans. Amer. Math.
  Soc. \textbf{280} (1983), no.~2, 673--693. \MR{MR716844 (85e:55029)}

\bibitem[Sin89]{singer_transfer}
\bysame, \emph{The transfer in homological algebra}, Math. Z. \textbf{202}
  (1989), no.~4, 493--523. \MR{1022818 (90i:55035)}

\bibitem[Sin78]{singer_loops_I}
\bysame, \emph{Iterated loop functors and the homology of the {S}teenrod
  algebra}, J. Pure Appl. Algebra \textbf{11} (1977/78), no.~1--3, 83--101.
  \MR{MR0478155 (57 \#17644)}

\bibitem[Wil83]{wilk}
Clarence Wilkerson, \emph{A primer on the {D}ickson invariants}, Proceedings of
  the {N}orthwestern {H}omotopy {T}heory {C}onference ({E}vanston, {I}ll.,
  1982), Contemp. Math., vol.~19, Amer. Math. Soc., Providence, RI, 1983,
  pp.~421--434. \MR{711066 (85c:55017)}

\bibitem[Zar84]{zarati_these}
Sa{\"{\i}}d Zarati, \emph{Dérivés du foncteur de déstabilisation en
  caractéristique impaire et applications}, Thèse d'état, Université Paris
  Sud, 1984.

\bibitem[Zar90]{zarati}
\bysame, \emph{Derived functors of the destabilization and the {A}dams spectral
  sequence}, Ast\'erisque (1990), no.~191, 8, 285--298, International
  Conference on Homotopy Theory (Marseille-Luminy, 1988). \MR{MR1098976
  (92c:55020)}

\end{thebibliography}

\providecommand{\bysame}{\leavevmode\hbox to3em{\hrulefill}\thinspace}
\providecommand{\MR}{\relax\ifhmode\unskip\space\fi MR }
\providecommand{\MRhref}[2]{%
  \href{http://www.ams.org/mathscinet-getitem?mr=#1}{#2}
}
\providecommand{\href}[2]{#2}

\end{document}